\documentclass[reqno]{amsart}
\usepackage{amsmath,amssymb,cite}
\usepackage[mathscr]{euscript}
\usepackage{graphics,epsfig,color}
\usepackage{orcidlink}
\usepackage[title]{appendix}
\theoremstyle{plain}
\newtheorem{theorem}{Theorem}[section]
\newtheorem{proposition}[theorem]{Proposition}
\newtheorem{lemma}[theorem]{Lemma}

\theoremstyle{definition}

\theoremstyle{remark}
\newtheorem{remark}{Remark}[section]

\numberwithin{equation}{section}
\numberwithin{theorem}{section}
\numberwithin{figure}{section}

\newcommand{\mc}[1]{{\mathcal #1}}

\newcommand{\bs}[1]{{\boldsymbol #1}}
\newcommand{\bb}[1]{{\mathbb #1}}

\newcommand{\rme}{\mathrm{e}}

\newcommand{\rmd}{\mathrm{d}}
\newcommand{\eps}{\varepsilon}

\newcommand{\id}{{1 \mskip -5mu {\rm I}}}
\newcommand{\supp}{\mathop{\rm supp}\nolimits}

\title[Leapfrogging vortex rings as scaling limit of Euler Equations]{Leapfrogging vortex rings as scaling limit of Euler Equations}

\author[P.\ Butt\`a]{Paolo Butt\`a \orcidlink{0000-0002-3193-4282}}
\address{Dipartimento di Matematica\\
Sapienza Universit\`a di Roma\\
P.le Aldo Moro 5, 00185 Roma\\
Italy}

\email{butta@mat.uniroma1.it}

\author[G.\ Cavallaro]{Guido Cavallaro \orcidlink{0000-0001-5668-370X}}
\address{Dipartimento di Matematica\\
Sapienza Universit\`a di Roma\\
P.le Aldo Moro 5, 00185 Roma\\
Italy}

\email{cavallar@mat.uniroma1.it}

\author[C.\ Marchioro]{Carlo Marchioro}
\address{Dipartimento di Matematica\\
Sapienza Universit\`a di Roma\\
P.le Aldo Moro 5, 00185 Roma\\
Italy\\
and\\
International Research Center M\&MOCS\\ 
Universit\`a di L'Aquila\\
Palazzo Caetani\\
04012 Cisterna di Latina (LT)\\
Italy}

\email{marchior@mat.uniroma1.it}

\keywords{Incompressible Euler flow, vortex rings, leapfrogging.}

\subjclass{
76B47, %Vortex flows for incompressible inviscid fluids
37N10. %Dynamical systems in  fluid mechanics, oceanography and meteorology
}

\date{}

\begin{document}

\begin{abstract}
We consider an incompressible fluid with axial symmetry without swirl, assuming initial data such that the initial vorticity is very concentrated inside $N$ small disjoint rings of thickness $\eps$, each one of vorticity mass and main radius of order $|\log\eps|$. When $\eps \to 0$, we show that, at least for small but positive times, the motion of the rings converges to a dynamical system firstly introduced in \cite{MN}. In the special case of two vortex rings with large enough main radius, the result is improved reaching longer times, in such a way to cover the case of several overtakings between the rings, thus providing a mathematical rigorous derivation of the leapfrogging dynamics.
\end{abstract}

\maketitle

\section{Introduction}
\label{sec:1}

We study the time evolution of an incompressible non viscous fluid in the whole space $\bb R^3$, in case of axial symmetry without swirl, when the initial vorticity is supported and sharply concentrated in $N$ annulii of large radius (i.e., distance from the symmetry axis) of leading term $\alpha|\log\eps|$ ($\alpha>0$ fixed), thickness of order $\eps$, vorticity mass of order $|\log\eps|$, and finite distance from each other. We are interested in considering the time evolution of such configuration in the limit $\eps \to 0$. In a previous paper of some years ago, \cite{MN}, the same problem was investigated, showing that for $N=1$ the vorticity remains concentrated for $t>0$ in an annulus with the same distance from the symmetry axis and thickness $\rho(\eps)$ (with $\rho(\eps)\to 0$ as $\eps\to 0$), moving with a constant speed along the symmetry axis. The case in which many coaxial vortex rings interact each other remained an open problem, and it was conjectured in \cite{MN} that in the limit $\eps\to 0$ the motion of the rings (parameterized throughout  suitable cylindrical coordinates) converges to the following dynamical system, which is the composition of the well-known point vortex system with a drift term along the symmetry axis,
\begin{equation}
\label{ode-intro}
\dot \zeta^i = -\frac{1}{2\pi} \sum_{j\ne i} a_j \frac{(\zeta^i-\zeta^j)^\perp}{|\zeta^i-\zeta^j|^2} + \frac{a_i}{4\pi\alpha} \begin{pmatrix} 1 \\ 0 \end{pmatrix}, \quad i=1,\ldots,N,
\end{equation}
where $\zeta^i= (\zeta^i_1, \zeta^i_2)\in \bb R^2$  ($(v_1,v_2)^\perp = (v_2,-v_1)$) and the real quantity $a_i$ is related to the vorticity mass of the $i$-th ring. This dynamical system accounts for an old observation regarding the so-called \textit{leapfrogging} phenomenon, which goes back to the work of Helmholtz \cite{H,H1}, who describes such configuration, in the case of two rings solely, with the following words \cite[p. 510]{H1}:

\textit{We can now see generally how two ring-formed vortex-filaments having the same axis would mutually affect each other, since each, in addition to its proper motion, has that of its elements of fluid as produced by the other. If they have the same direction of rotation, they travel in the same direction; the foremost widens and travels more slowly, the pursuer shrinks and travels faster till finally, if their velocities are not too different, it overtakes the first and penetrates it. Then the same game goes on in the opposite order, so that the rings pass through each other alternately.}

Indeed, as discussed in Section \ref{sec:7}, Eq.~\eqref{ode-intro} admits solutions such that the relative position $\zeta^1-\zeta^2$ performs a periodic motion, which corresponds to the leapfrogging motion of the rings.

Even if this phenomenon has been known since Helmholtz, addressed in many papers, such as \cite{aiki, bori, buffoni, da, dyson, dyson1, hicks, klein, lamb}, and studied also from a numerical point of view \cite{chen, lim, naka, riley}, its mathematical justification, as a rigorous derivation from Euler equation, has received only recently a positive answer in \cite{DDMW}, in which it is constructed a special solution exhibiting this feature (see also \cite{JS, js2021} in the context of the Gross-Pitaevskii equation).

The dynamics of several coaxial vortex rings at distance of order $|\log\eps|$ from the symmetry axis appears to represent a critical regime, in which the two terms on the right-hand side of Eq.~\eqref{ode-intro} are of the same order (the first one is the interaction with the other vortex rings, the second one is the self-induced field). The motion exhibits features different from the cases in which the distance is of order $|\log\eps|^k$, $0\le k<1$, or the distance is of order $|\log\eps|^k$, $k>1$ (or larger). Note that when $0\le k<1$, in order to have the self-induced field not diverging, the vorticity mass of each ring has to be chosen of order $|\log\eps|^{2k-1}$. In this case, for $k=0$, the self-induced field acting on each ring is dominant with respect to its interaction with the others, and the rings perform rectilinear motions with constant speed, see \cite{BCM, BuM2}, while for $0<k<1$ the dynamics has not been studied explicitly but we believe that the behavior is analogous to the case $k=0$. When the distance is of order $|\log\eps|^k$, $k>1$ (or larger), the interaction of each vortex ring with the others is dominant with respect to the self-induced field and the motion is described by the point vortex system \cite{CS, CavMarJMP,Mar99} (explicitly studied for $k>2$ in \cite{CavMarJMP}, while for $1<k\le2$ we believe to get the same behavior). We remark that for $k>1$ the vorticity mass of each ring has to be chosen of order $|\log\eps|^k$ to have a not trivial behavior.

In the present paper we analyze the critical regime and prove that the aforementioned conjecture of \cite{MN} holds true. More precisely, we show that in the limit $\eps\to 0$ (when the vorticity becomes very large) and for quite general initial data the motion of the rings is governed by Eq.~\eqref{ode-intro}, at least for short but positive times. This is indeed the main difference with respect to the result obtained in \cite{DDMW}, since while we study the Cauchy problem, with arbitrary initial data (we require that the initial vorticity is bounded and supported on separated rings), in \cite{DDMW} a special solution is constructed. Moreover in \cite{DDMW} the authors adopt a different scaling with respect to the present one, that is  a distance $O(|\log\eps|^{-\frac12})$ between the rings and $O(1)$ from the symmetry axis  (they need to scale also the time).  The two scalings are not in contradiction, since they both give rise to almost the same limiting dynamical system, where the competitor terms are of the same order (the interaction between the rings and the drift term deriving from the axial symmetry).  Our scaling is preferable for our techniques, since it allows us to control the interaction between the rings. Lastly, the methods used in  \cite{DDMW} are very different from ours, in \cite{DDMW} the authors use a  \textit{gluing} scheme for PDE, while our approach is based on an accurate use of conserved quantities. Our proof indeed relies on some new ideas, merged with methods of previous papers \cite{MN, BCM00, BuM2}, which allow to overcome delicate technical points that seemed hard to be solved (since the problem was stated, in the present framework, in \cite{MN}). Going into specifics, we detail below the main points.

\noindent
($i$) The fact that the energy of each vortex ring almost conserves its initial value, at the leading term, allows to state that the vorticity mass of each vortex ring is concentrated inside a torus (whose cross section has a diameter vanishing with $\eps$) during the time evolution. This result could at first sight seem not obvious, considering that a plain estimate of the time derivative of the energy of each vortex ring implies \textit{a priori} a variation of the same order of the initial energy. This is the content of Section \ref{sec:3}.

\noindent
($ii$) The iterative method, used to prove that each vortex ring has compact support at positive times (which is an essential tool to control the interaction among the vortex rings), requires to be splitted into two separated procedures, due to the fact that the a priori estimate of a fundamental quantity, the moment of inertia, is not good enough to make work the iterative method in its standard form (as used, for example, in \cite{BuM2}). This is done in Section \ref{sec:4}, while the subsequent support property is proved in Section \ref{sec:5}.

Once concentration and localization properties of the vortex rings are guaranteed, the proof of the main theorem on the convergence to the system Eq.~\eqref{ode-intro} can be easily concluded. This is the content of Section \ref{sec:6}.

The last section of the paper, Section \ref{sec:7}, concerns the leapfrogging phenomenon, treated in the special case of two rings discussed by Helmholtz. The convergence result, as stated in general and applied in this context for an appropriate choice of the initial data, guarantees at most one overtaking between the rings within the time interval of convergence. This is not completely satisfactory since, in accordance with experimental and numerical observations, several overtakings can take place before the rings dissolve and lose their shape \cite{YM,AN}.

Fortunately, in the special case of two vortex rings with large enough main radii, we can extend the time of convergence in order to cover several crossings between the rings. More precisely, it is possible to repeatedly apply the construction of item (i)-(ii) by suitably increasing the parameter $\alpha$ (i.e., the distance from the symmetry axis), thus reaching any arbitrarily fixed time. This is not really surprising, since as $\alpha$ increases the system approaches (formally) a planar fluid,  Eq.~\eqref{ode-intro} gets closer to the standard point vortex model, and, in the planar case, convergence to the point vortex model occurs globally in time (even up to times diverging with $\eps$, see \cite{BuM1}). 

\section{Notation and main result}
\label{sec:2}

The Euler equations governing the time evolution in three dimension of an incompressible inviscid fluid of unitary density with velocity $\bs u = \bs u(\bs\xi,t)$ decaying at infinity take the form,
\begin{gather}
\label{vorteq}
\partial_t \bs\omega + (\bs u\cdot \nabla) \bs\omega  = (\bs \omega\cdot \nabla) \bs u\,,  
\\ \label{u-vort} 
\bs u(\bs\xi,t) = - \frac{1}{4\pi} \int_{\bb R^3}\! \rmd \bs\eta \, \frac{(\bs\xi-\bs\eta) \wedge \bs\omega(\bs\eta,t)}{|\bs\xi-\bs\eta|^3} \,,
\end{gather}
where $\bs\omega = \bs \omega(\bs\xi,t) = \nabla \wedge \bs u(\bs\xi,t)$ is the vorticity, $\bs\xi = (\xi_1,\xi_2,\xi_3)$ denotes a point in $\bb R^3$, and $t\in \bb R_+$ is the time. Note that Eq.~\eqref{u-vort} clearly implies the incompressibility condition $\nabla\cdot \bs u=0$.

Our analysis is restricted to the special class of axisymmetric (without swirl) solutions to Eqs.~\eqref{vorteq}, \eqref{u-vort}. We recall that a vector field $\bs F$ is called axisymmetric without swirl if, denoting by the $(z,r,\theta)$ the cylindrical coordinates in a suitable frame, the cylindrical components $(F_z, F_r, F_\theta)$ of $\bs F$ are such that $F_\theta=0$ and both $F_z$ and $F_r$ are independent of $\theta$. 

The axisymmetry is preserved by the time evolution. Furthermore, when restricted to axisymmetric velocity fields $\bs u(\bs\xi,t) = (u_z(z,r,t), u_r(z,r,t), 0)$, the vorticity is
\begin{equation}
\label{omega}
\bs\omega = (0,0,\omega_\theta) = (0,0,\partial_z u_r - \partial_r u_z)
\end{equation}
and, denoting henceforth $\omega_\theta$ by $\Omega$, Eq.~\eqref{vorteq} reduces to
\begin{equation}
\label{omeq}
\partial_t \Omega + (u_z\partial_z + u_r\partial_r) \Omega - \frac{u_r\Omega}r = 0 \,,
\end{equation}
Finally, from Eq.~\eqref{u-vort}, $u_z = u_z(z,r,t)$ and $u_r=u_r(z,r,t)$ are given by
\begin{align}
\label{uz}
u_z & = - \frac{1}{2\pi} \int\! \rmd z' \!\int_0^\infty\! r' \rmd r' \! \int_0^\pi\!\rmd \theta \, \frac{\Omega(z',r',t) (r\cos\theta - r')}{[(z-z')^2 + (r-r')^2 + 2rr'(1-\cos\theta)]^{3/2}} \,,
\\ \label{ur}
u_r & = \frac{1}{2\pi} \int\! \rmd z' \!\int_0^\infty\! r' \rmd r' \! \int_0^\pi\!\rmd \theta \, \frac{\Omega(z',r',t) (z - z')\cos\theta}{[(z-z')^2 + (r-r')^2 + 2rr'(1-\cos\theta)]^{3/2}} \,.
\end{align}
Otherwise stated, the axisymmetric solutions to the Euler equations are given by the solutions to Eqs.\ \eqref{omeq}, \eqref{uz}, and \eqref{ur}. We also notice that the incompressibility condition reduces to
\begin{equation}
\label{divu0}
\partial_z(ru_z) + \partial_r(ru_r) = 0\,.
\end{equation}

Since we are interested also to non-smooth initial data, we shall consider weak formulations of the equations of motion. To this end, we notice that Eq.~\eqref{omeq} expresses that the quantity $\Omega/r$ is conserved along the flow generated by the velocity field, i.e., 
\begin{equation}
\label{cons-omr}
\frac{\Omega(z(t),r(t),t)}{r(t)} = \frac{\Omega(z(0),r(0),0)}{r(0)} \,,
\end{equation}
with $(z(t),r(t))$ solution to
\begin{equation}
\label{eqchar}
\dot z(t) = u_z(z(t),r(t),t) \,, \qquad \dot r(t) = u_r(z(t),r(t),t) \,.
\end{equation}
Eqs.~\eqref{uz}, \eqref{ur}, \eqref{cons-omr}, and \eqref{eqchar} can be assumed as a weak formulation of the Euler equations in the framework of axisymmetric solutions. An equivalent weak formulation is still obtained from Eq.~\eqref{omeq} by a formal integration by parts,
\begin{equation}
\label{weqO}
\frac{\rmd}{\rmd t} \Omega_t[f] = \Omega_t[u_z\partial_z f + u_r\partial_r f + \partial_t f] \,,
\end{equation}
where $f = f(z,r,t)$ is any bounded smooth test function and 
\[
\Omega_t[f] := \int\! \rmd z \!\int_0^\infty\! \rmd r \, \Omega(z,r,t) f(z,r,t) \,.
\]

The existence of a global solution both for the Euler and Navier-Stokes equations has been established many years ago \cite{Lad,UY}, see also \cite{Fe-Sv,Gal13,Gal12} for more recent results. Global in time existence and uniqueness of a weak solution to the related Cauchy problem holds when the initial vorticity is a bounded function with compact support contained in the open half-plane $\Pi:=\lbrace(z,r):r>0\rbrace$, see, e.g.,  \cite[Page 91]{MaP94} or \cite[Appendix]{CS}. In particular, the support of the vorticity remains in the open half-plane $\Pi$ at any time. 

We choose initial data representing a system of $N$ concentrated vortex rings, each one with cross-section of radius not larger than $\eps$ and main radius (i.e., distance from the symmetry axis) of order $|\log\eps|$, where $\eps\in (0,1)$ is a small parameter. More precisely, denoting by $\Sigma(\zeta|\rho)$ the disk of center $\zeta$ and radius $\rho$, we fix $\alpha>0$, $N$ distinct points $\zeta^i \in \bb R^2$, $i=1,\ldots,N$, and $\eps_0$ small enough to have
\begin{gather*}
\overline{\Sigma((0,r_\eps) + \zeta^i|\eps)} \subset\Pi\quad \forall\, i  \quad  \forall\, \eps\in (0,\eps_0)\,, \\ \Sigma((0,r_\eps)  + \zeta^i|\eps)\cap \Sigma((0,r_\eps)  + \zeta^j|\eps)=\emptyset\quad \forall\, i\ne j  \quad  \forall\, \eps\in (0,\eps_0)\,,
\end{gather*}
where
\begin{equation}
\label{reps}
r_\eps = \alpha|\log\eps|\,.
\end{equation}
We then choose
\begin{equation}
\label{inO}
\Omega_\eps(z,r,0) = \sum_i \Omega_{i,\eps}(z,r,0) \quad \forall\,\eps\in (0,\eps_0)\,,
\end{equation}
where each $\Omega_{i,\eps}(z,r,0)$ is a non-negative or non-positive function such that
\[
\supp\, \Omega_{i,\eps}(\cdot,0) \subset \Sigma((r_\eps,0) + \zeta^i|\eps)\quad \forall\,\eps\in (0,\eps_0)\,.
\]

We also assume that there are $N$ real parameters $a_1,\ldots, a_N$, called the \textit{vortex intensities}, such that
\[
\int\!\rmd z \!\int_0^\infty\!\rmd r\, \Omega_{i,\eps}(z,r,0) = a_i \quad \forall\, \eps\in (0,\eps_0)\,,
\]
which means that the vorticity mass of each ring is proportional to its mean radius, i.e,  order $r_\eps$. Finally, to avoid too large vorticity concentrations, our last assumption is the existence of a constant $M>0$ such that 
\[
|\Omega_{i,\eps}(z,r,0)| \le \frac{M}{\eps^2} \quad \forall\, \eps\in (0,\eps_0)\,.
\]

In view of Eq.~\eqref{cons-omr}, the decomposition Eq.~\eqref{inO} extends to positive time setting
\[
\Omega_\eps(z,r,t) = \sum_i \Omega_{i,\eps}(z,r,t) \quad \forall\,\eps\in (0,\eps_0)\,,
\]
with $\Omega_{i,\eps}(x,t)$ the time evolution of the $i$-th vortex ring,
\[
\Omega_{i,\eps}(z(t),r(t),t) := \frac{r(t)}{r(0)} \Omega_{i,\eps}(z(0),r(0),0)\,.
\]

Since the parameter $\eps$ will eventually go to zero, it is convenient to introduce the new variables $x = (x_1,x_2)$ defined by
\[
z = x_1\,, \quad r = r_\eps + x_2\,.
\]
It is also useful to extend the vorticity expressed in these new variables to a function on the whole plane. More precisely, we define 
\begin{equation}
\label{Oo}
\omega_\eps(x,t) = \begin{cases} \Omega_\eps(x_1,r_\eps+x_2,t) & \text{if } x_2 > -r_\eps\,, \\ 0 &\text{if } x_2 \le -r_\eps\,, \end{cases}
\end{equation}
and the same position defines $\omega_{i,\eps}(x,t)$ provided $\Omega_\eps$ is replaced by $\Omega_{i,\eps}$ in the right-hand side. In particular, with a slight abuse of notation, we shall write 
\begin{gather*}
\int\! \rmd z \!\int_0^\infty\! \rmd r \, \Omega_\eps(z,r,t) G(z,r) = \int\!\rmd x\, \omega_\eps(x,t) G(x_1,r_\eps + x_2)\,, \\ 
\int\! \rmd z \!\int_0^\infty\! \rmd r \, \Omega_{i,\eps}(z,r,t) G(z,r) = \int\!\rmd x\, \omega_{i,\eps}(x,t) G(x_1,r_\eps + x_2)\,,
\end{gather*}
despite a function $x\mapsto G(x_1,r_\eps+x_2)$ is defined only if $x_2> - r_\eps$.

In this way, the equations of motion Eqs.~\eqref{uz}, \eqref{ur}, \eqref{cons-omr}, and \eqref{eqchar} take the following form,
\begin{equation}
\label{u=}
u(x,t) = \int\!\rmd y\, H(x,y)\, \omega_\eps(y,t)\,,
\end{equation}
\begin{equation}
\label{cons-omr_n}
\omega_\eps(x(t),t) = \frac{r_\eps + x_2(t)}{r_\eps + x_2(0)} \omega_\eps(x(0),0) \,, 
\end{equation}
\begin{equation}
\label{eqchar_n}
\dot x(t) = u(x(t),t) \,,
\end{equation}
where $u(x,t) = (u_1(x,t), u_2(x,t))$ and the kernel $H(x,y) = (H_1(x,y),H_2(x,y))$ is given by
\begin{align}
\label{H1}
H_1(x,y) & = \frac{1}{2\pi} \int_0^\pi\!\rmd \theta \, \frac{(r_\eps + y_2)(r_\eps + y_2 - (r_\eps + x_2)\cos\theta)}{\big[|x-y|^2 + 2(r_\eps + x_2) (r_\eps + y_2) (1-\cos\theta)\big]^{3/2}} \,,
\\ \label{H2}
H_2(x,y) & = \frac{1}{2\pi} \int_0^\pi\!\rmd \theta \, \frac{(r_\eps +y_2) (x_1-y_1) \cos\theta}{\big[|x-y|^2 + 2(r_\eps + x_2)(r_\eps + y_2)(1-\cos\theta)\big]^{3/2}} \,.
\end{align}
(we omit the explicit dependence of $u$ and $H$ on $\eps$). Moreover, the initial data Eq.~\eqref{inO} now reads,
\begin{equation}
\label{in}
\omega_\eps(x,0) = \sum_i  \omega_{i,\eps}(x,0) \quad \forall\,\eps\in (0,\eps_0)\,,
\end{equation}
with $\omega_{i,\eps}(x(0),0)$ satisfying
\begin{gather}
\label{initial}
\Lambda_{i,\eps}(0) := \supp\, \omega_{i,\eps}(\cdot,0) \subset \Sigma(\zeta^i|\eps) \quad \forall\,\eps\in (0,\eps_0)\,, 
\\ \label{ai} 
\int\!\rmd x \, \omega_{i,\eps}(x,0) = a_i \quad\forall\, \eps\in (0,\eps_0)\,, 
\\ \label{Mgamma}
|\omega_{i,\eps}(x,0)| \le \frac{M}{\eps^2} \quad\forall\, \eps\in (0,\eps_0)\,.
\end{gather}
Finally,
\begin{equation}
\label{in-t}
\omega_\eps(x,t) = \sum_i \omega_{i,\eps}(x,t) \quad \forall\,\eps\in (0,\eps_0)\,,
\end{equation}
with
\begin{equation}
\label{cons-omr_ni}
\omega_{i,\eps}(x(t),t) =  \frac{r_\eps + x_2(t)}{r_\eps + x_2(0)} \omega_{i,\eps}(x(0),0)\,,
\end{equation}
where $x(t)$ solves Eq.~\eqref{eqchar_n}. It follows that each $\omega_{i,\eps}(x,t)$ remains non-negative or non-positive also for $t>0$. Moreover, the weak formulation Eq.~\eqref{weqO} holds also separately for each $\omega_{i,\eps}(x,t)$, and reads
\begin{equation}
\label{weq}
\frac{\rmd}{\rmd t} \int\!\rmd x\, \omega_{i,\eps}(x,t) f(x,t) = \int\!\rmd x\, \omega_{i,\eps}(x,t) \big[ u \cdot \nabla f + \partial_t f \big](x,t) \,.
\end{equation}
In particular, the vortex intensities are conserved during  the time evolution,
\begin{equation}
\label{ait}
M_0^i(t) := \int\!\rmd x \, \omega_{i,\eps}(x,t) = a_i \quad \forall\, t \ge 0 \quad\forall\, \eps\in (0,\eps_0)\,.
\end{equation}

Sometimes, in the sequel, we will improperly call vorticity mass of a vortex ring its intensity. More generally, the quantity $\int_D\!\rmd x\, \omega_{i,\eps}(x,t)$ will be indicated as the amount of vorticity mass of the $i$-th vortex ring contained in the region $D\subseteq \bb R^2$.

We now denote by $(\zeta^1(t), \ldots,\zeta^N(t))$, $t\in [0,T^*)$, the maximal solution to the Cauchy problem,
\begin{equation}
\label{ode}
\begin{cases} \dot \zeta^i(t) = {\displaystyle  \sum_{j\ne i} a_j K(\zeta^i(t)-\zeta^j(t)) +\frac{a_i}{4\pi\alpha} \begin{pmatrix} 1 \\ 0 \end{pmatrix}} \\   \zeta^i(0) = \zeta^i \end{cases} \quad  \forall\, i=1, \ldots,N\,,
\end{equation}
with $\{\zeta^i\}_{i=1}^N$ as in Eq.~\eqref{initial} and
\begin{equation}
\label{Kern}
K(x) := - \frac{1}{2\pi} \nabla^\perp \log|x|
\end{equation}
(here, if $v = (v_1,v_2)$ then $v^\perp = (v_2,-v_1)$).

We can now state the main result of the paper.

\begin{theorem}
\label{thm:1}
Assume the initial condition $\omega_\eps(x,0)$ verifies Eqs.~\eqref{in}, \eqref{initial}, \eqref{ai}, and \eqref{Mgamma}. Then, for any fixed (independent of $\eps$) $\varrho>0$  such that the closed disks $\overline{\Sigma(\zeta^i|2\varrho)}$ are mutually disjointed, there exists $T_\varrho \in (0,T^*)$ such that for any $\eps$ small enough and $t\in [0,T_\varrho]$ the following holds true.

\begin{itemize}
\item[(1)] $\Lambda_{i,\eps}(t) := \supp\, \omega_{i,\eps}(\cdot,t) \subseteq \Sigma(\zeta^i(t)|\varrho)$ and the disks $\Sigma(\zeta^i(t)|2\varrho)$ are mutually disjointed.
\item[(2)] There exist $(\zeta^{1,\eps}(t), \ldots,\zeta^{N,\eps}(t))$ and $\varrho_\eps>0$ such that
\[
\lim_{\eps\to 0} \int_{\Sigma(\zeta^{i,\eps}(t)|\varrho_\eps)}\!\rmd x\, \omega_{i,\eps}(x,t) = a_i\quad \forall\, i=1,\ldots,N,
\]
with $\displaystyle \lim_{\eps\to 0} \varrho_\eps = 0$, and
\[
\lim_{\eps\to 0} \zeta^{i,\eps}(t) = \zeta^i(t) \quad  \forall\, i=1,\ldots,N.
\]
\end{itemize}
\end{theorem}

The time interval of convergence can be enlarged in the case of two vortex rings with initial data such that the relative position $\zeta^1-\zeta^2$ performs a periodic motion (with respect to the evolution Eq.~\eqref{ode-intro} with $N=2$) and $\alpha$ is chosen sufficiently large. For brevity in the exposition, we do not detail the result here and address the reader to Section \ref{sec:7}.

\section{Concentration estimates}
\label{sec:3}

Given $\varrho$ as in the statement of Theorem \ref{thm:1}, since $|\zeta^i-\zeta^j| > 4\varrho$ for any $i\ne j$, we can find $T\in (0,T^*)$ such that
\begin{equation}
\label{T}
\min_{i\ne j} \min_{t\in [0,T]} |\zeta^i(t)-\zeta^j(t)| \ge 4\varrho\,,
\end{equation}
and let also
\begin{equation}
\label{dbarra}
\bar d:= \max_i\max_{t\in [0,T]} |\zeta^i(t)|\,.
\end{equation}
We then define
\begin{equation}
\label{Teps}
T_\eps := \max\left\{t\in [0,T] \colon \Lambda_{i,\eps}(s) \subset \Sigma(\zeta^i(s)|\varrho) \; \forall\, s\in [0,t]\; \forall\, i \right\}.
\end{equation}
Without loss of generality, hereafter we assume $\eps_0 < \varrho$ so that $T_\eps>0$ for any $\eps\in (0,\eps_0)$ by continuity. Clearly,
\begin{equation}
\label{sep-disks}
\begin{split}
& |x| \le \bar d + \varrho \quad  \forall\, x\in \Sigma(\zeta^i(t)|\varrho)  \quad  \forall\, t \in [0,T] \quad \forall\, i \,, \\ & |x-y| \ge 2\varrho \quad \forall\, x\in\Sigma(\zeta^i(t)|\varrho) \quad\forall\, y \in\Sigma(\zeta^j(t)|\varrho) \quad \forall\, t \in [0,T] \quad \forall\, i\ne j\,,
\end{split}
\end{equation}
and therefore, up to time $T_\eps$, also the supports of the vortex rings are uniformly bounded and separated,
\begin{equation}
\label{sep-supp}
\begin{split}
& |x| \le \bar d + \varrho \quad  \forall\, x\in \Lambda_{i,\eps}(t) \quad  \forall\, t \in [0,T_\eps] \quad \forall\, i \,, \\ & |x-y| \ge 2\varrho \quad \forall\, x\in \Lambda_{i,\eps}(t)\quad\forall\, y \in \Lambda_{j,\eps}(t) \quad \forall\, t \in [0,T_\eps] \quad \forall\, i\ne j\,.
\end{split}
\end{equation}
Clearly, $T_\eps$ could vanish as $\eps\to 0$, the key point in proving Theorem \ref{thm:1} will be a bootstrap argument based on the analysis of the motion in the time interval $[0,T_\eps]$ which shows that in fact this is not the case. The first ingredient for such analysis are suitable concentration inequalities on the vorticities, which are the content of the present section.

In \cite{MN} and previous work \cite{BCM00}, concentration estimates on the vorticity mass in the case of a single vortex ring are deduced by using the conservation laws of kinetic energy, axial moment of inertia, and vortex intensity. Here, we need similar estimates for the vorticity of each vortex ring. The corresponding kinetic energies and axial moments of inertia are not conserved, but as long as the interaction among the rings is not too large, i.e., up to time $T_\eps$, it is still possible to obtain such inequalities.

\subsection{Energy estimates}
\label{sec:3.1}

The kinetic energy $E(t) =  \frac 12 \int\!\rmd\bs\xi\, |\bs u (\bs\xi,t)|^2$ associated to axisymmetric solutions described via Eqs.~\eqref{u=}, \eqref{cons-omr_n}, and \eqref{eqchar_n} takes the form
\[
E(t) = \frac 12 \int\!\rmd z \!\int_0^\infty\!\rmd r\, 2\pi r\, (u_z^2 + u_r^2)(z,r,t) =  \frac 12 \int\! \rmd x \,  2\pi (r_\eps + x_2) |u(x,t)|^2\,.
\]
It is convenient to express $E(t)$ as a quadratic form of the vorticity $\omega_\eps(x,t)$. To this end, we introduce the stream function
\[
\Psi(x,t) = \int\!\rmd y\, S(x,y)\, \omega_\eps(y,t)\,, 
\]
where the Green kernel $S(x,y)$ reads
\[
S(x,y) := \frac{(r_\eps + x_2) (r_\eps + y_2)}{2\pi} \int_0^\pi\!\rmd\theta\, \frac{\cos\theta}{\sqrt{|x-y|^2 + 2(r_\eps + x_2) (r_\eps + y_2)(1-\cos\theta)}}\,,
\]
so that $u(x,t) = (r_\eps + x_2)^{-1} \nabla^\perp\Psi(x,t)$ and the energy takes the form (see, e.g., \cite{BCM00,F})
\[
E(t) = \pi \int\! \rmd x\, \Psi(x,t) \, \omega_\eps(x,t) = \pi \int\!\rmd x \int\!\rmd y\, S(x,y)\, \omega_\eps(x,t) \, \omega_\eps(y,t)\,.
\]
In view of Eq.~\eqref{in-t}, the energy can be decomposed as the sum of the energies due to the self-interaction of each vortex ring plus those due to the interaction among the rings,
\begin{gather}
\label{E=sum Ei}
E(t) = \sum_i E_i(t) + 2 \sum_{i > j} E_{i,j}(t)\,, \\ \label{Ei=} 
E_i(t) = \pi \int\!\rmd x \int\!\rmd y\, S(x,y)\, \omega_{i,\eps}(x,t) \, \omega_{i,\eps}(y,t)\,,
\\ \label{Eij=}
E_{i,j}(t) = \pi \int\!\rmd x \int\!\rmd y\, S(x,y)\, \omega_{i,\eps}(x,t) \, \omega_{j,\eps}(y,t)\,.
\end{gather}
Hereafter, we let
\begin{equation}
\label{|a|}
|a| = \sum_i |a_i|\,.
\end{equation}

\begin{lemma}
\label{lem:E>}
There exists $C_1 = C_1(\alpha, |a|, \bar d, \varrho)>0$ such that, for any $\eps \in (0,\eps_0)$,
\begin{equation}
\label{E>}
\sum_i E_i(t) \ge \frac\alpha2 \sum_i a_i^2 |\log\eps|^2 - C_1 (\log|\log\eps|)|\log\eps| \quad \forall\, t\in [0,T_\eps]\,.
\end{equation}
\end{lemma}

\begin{proof}
We write
\begin{equation}
\label{S=I}
S(x,y) = \frac{\sqrt{(r_\eps + x_2) (r_\eps + y_2)}}{2\pi} I_0\left(\frac{|x-y|}{\sqrt{(r_\eps + x_2) (r_\eps + y_2)}}\right),
\end{equation}
where
\[
I_0(s) := \int_0^\pi\!\rmd\theta\, \frac{\cos\theta}{[s^2 + 2(1-\cos\theta)]^{1/2}}\,, \quad s>0\,,
\]
can be easily evaluated, see, e.g., \cite[Appendix A]{MN}, getting
\begin{equation}
\label{I1}
C_0 := \sup_{s>0} \left|I_0(s) - \log \frac{2+\sqrt{s^2 + 4}}s \right| < \infty\,.
\end{equation}

From Eqs.~\eqref{Eij=}, \eqref{S=I}, \eqref{I1} and recalling Eqs.~\eqref{sep-supp}, \eqref{reps}, and \eqref{ait}, it follows that there is $C_1'=C_1'(\alpha,\bar d,\varrho)>0$ such that
\[
|E_{i,j}(t)| \le  C_1'|a_i|\, |a_j| (\log|\log\eps|)|\log\eps| \quad \forall\, t\in [0,T_\eps] \quad \forall\, \eps\in (0,\eps_0)\quad \forall\, i\ne j\,.
\]
Analogously, in view of Eq.~\eqref{initial}, there is $C_1'' = C_1''(\alpha,\bar d)>0$ such that
\[
E_i(0) \ge a_i^2\left[ \frac{\alpha}2 |\log\eps|^2 - C_1'' (\log|\log\eps|)|\log\eps| \right] \quad \forall\, \eps\in (0,\eps_0) \quad \forall\, i\,.
\]
Therefore, since the total kinetic energy is conserved along the motion, i.e., $E(t) = E(0)$, we conclude that
\begin{align*}
\sum_i E_i(t) & = \sum_i E_i(0) + 2 \sum_{i>j} [E_{i,j}(0) - E_{i,j}(t) ] \\ & \ge \sum_i E_i(0) - 2 \sum_{i>j} (|E_{i,j}(0)| + |E_{i,j}(t)|) \ \quad \forall\, t\in [0,T_\eps] \quad \forall\,\eps\in (0,\eps_0)\,,
\end{align*}
from which Eq.~\eqref{E>} follows with, e.g., $C_1 = 2(C_1'+ C_1'') |a|^2$. 
\end{proof}

Without loss of generality, in what follows we further assume $\eps_0<1/\rme^\rme$, so that $\log|\log\eps| >1$ for any $\eps\in (0,\eps_0)$.

\begin{proposition}
\label{prop:conc1}
There exists $C_2 = C_2(\alpha,|a|,\bar d, \varrho)>0$ such that, for any $\eps \in (0,\eps_0)$,
\begin{align}
\label{conc1}
\mc G_i(t) & := \int\!\rmd x \!\int\! \rmd y\, \omega_{i,\eps}(x,t) \omega_{i,\eps}(y,t) \log\Big(\frac{|x-y|}{\eps}\Big) \id(|x-y| \ge \eps) \nonumber \\ & \qquad \le C_2 \log|\log\eps| \quad \forall\, t\in [0,T_\eps] \quad \forall\, i = 1,\ldots,N \,,
\end{align}
where $\id(\cdot)$ denotes the indicator function of a subset.
\end{proposition}

\begin{proof}
Letting
\begin{equation}
\label{Ab}
A := (r_\eps+x_2)(r_\eps+y_2)\,,
\end{equation}
from Eqs.~\eqref{S=I} and \eqref{I1} we have,
\[
S(x,y)  \le \frac{\sqrt A}{2\pi} \Big( C_0 + \log(\sqrt{4A} + \sqrt{|x-y|^2+4A}) - \log |x-y| \Big)
\]
and, in view of Eqs.~\eqref{Teps} and \eqref{sep-supp},
\begin{equation}
\label{A<}
r_\eps - \bar d - \varrho \le \sqrt A \le r_\eps+ \bar d + \varrho\,, \quad |x-y| \le 2\varrho  \quad  \forall\, x,y\in \Lambda_{i,\eps}(t) \quad  \forall\, t \in [0,T_\eps]\,.
\end{equation}
Therefore, recalling Eqs.~\eqref{reps} and \eqref{Ei=}, there exists $C_2' = C_2'(\alpha,|a|,\bar d, \varrho)>0$ such that
\[
\begin{split}
E_i(t) & \le \frac 12 \int\!\rmd x \!\int\! \rmd y\, \omega_{i,\eps}(x,t) \omega_{i,\eps}(y,t)  \sqrt A \log( |x-y|^{-1}) \nonumber \\ & \quad + C_2' (\log|\log\eps|)|\log\eps| \quad \forall\, t\in [0,T_\eps] \quad \forall\,\eps\in (0,\eps_0)\,,
\end{split}
\]
which can be recast as
\[
E_i(t) \le  G_i^{(1)}(t) - G_i^{(2)}(t)+ G_i^{(3)}(t)  + C_2' (\log|\log\eps|)|\log\eps| \quad \forall\, t\in [0,T_\eps] \quad \forall\,\eps\in (0,\eps_0)\,,
\]
where
\[
\begin{split}
G_i^{(1)}(t) & = \frac 12 \int\!\rmd x \!\int\! \rmd y\, \omega_{i,\eps}(x,t) \omega_{i,\eps}(y,t) \sqrt A \log\Big( \frac1\eps\Big)\,, \\ G_i^{(2)}(t) & = \frac 12 \int\!\rmd x \!\int\! \rmd y\, \omega_{i,\eps}(x,t) \omega_{i,\eps}(y,t) \sqrt A \log \Big(\frac{|x-y|}{\eps}\Big)\id(|x-y| \ge \eps) \,,  \\ G _i^{(3)}(t) & = \frac 12 \int\!\rmd x \!\int\! \rmd y\, \omega_{i,\eps}(x,t) \omega_{i,\eps}(y,t)\sqrt A \log\Big(\frac\eps{|x-y|} \Big) \id(|x-y| < \eps)\,.
\end{split}
\]
By Eq.~\eqref{A<} it follows that there is $C_2'' = C_2''(\alpha,|a|,\bar d, \varrho)>0$ such that
\[
G_i^{(1)}(t) \le \frac\alpha2 a_i^2 |\log\eps|^2 + C_2'' |\log\eps| \quad \forall\, t\in [0,T_\eps] \quad \forall\, \eps\in (0,\eps_0)\,.
\]
Concerning $G_i^{(3)}(t)$, again by Eq.~\eqref{A<} we have
\[
G_i^{(3)}(t) \le \frac 12(r_\eps+\bar d+\varrho)  \int\!\rmd x \, |\omega_{i,\eps}(x,t)| \int\! \rmd y\,|\omega_{i,\eps}(y,t)| \log\Big(\frac\eps{|x-y|} \Big) \id(|x-y| < \eps)\,,
\]
and the integral with respect to the variable $y$ can be estimated performing a symmetrical rearrangement of the vorticity around the point $x$. More precisely, by Eqs.~\eqref{Mgamma}, \eqref{initial}, \eqref{cons-omr_ni}, and \eqref{sep-supp},
\begin{equation}
\label{omt<}
|\omega_{i,\eps}(y,t)| \le \frac{r_\eps+\bar d+\varrho}{r_\eps-\eps} \frac{M}{\eps^2} \quad \forall\, t\in [0,T_\eps]\,,
\end{equation}
so that, by Eq.~\eqref{ait} and since $\omega_{\eps,i}(\cdot,t)$ does not change sign, if $\bar r$ is such that 
\begin{equation}
\label{br=}
\frac{r_\eps+\bar d+\varrho}{r_\eps-\eps} \frac{M}{\eps^2} \,\pi \bar r^2 = |a_i|
\end{equation}
then
\begin{gather*}
\int\! \rmd y\, |\omega_{i,\eps}(y,t)| \log\Big(\frac\eps{|x-y|} \Big) \id(|x-y| \le \eps) \le \frac{r_\eps+\bar d+\varrho}{r_\eps-\eps} \frac{M}{\eps^2} \int_0^{\bar r\wedge \eps}\!\rmd r\, 2\pi r \log\Big(\frac\eps r\Big) \\ = \frac{2 |a_i|}{\bar r^2}\left(\frac{(\bar r\wedge \eps)^2}4 - \frac{(\bar r\wedge \eps)^2}2\log\frac{(\bar r\wedge \eps)}\eps\right).
\end{gather*}
Hence, the above integral is bounded uniformly with respect to $\eps$. Therefore, again recalling Eq.~\eqref{reps}, there exists $C_2''' = C_2'''(\alpha,|a|,\bar d, \varrho)>0$ such that
\[
G_i^{(3)}(t) \le C_2''' |\log\eps|\,.
\]

Gathering together the above estimates, we conclude that
\[
\begin{split}
\sum_i E_i(t)  & \le \frac \alpha 2 \sum_i a_i^2 |\log\eps|^2 - \sum_i G_i^{(2)}(t) \\ & \quad + NC_2' (\log|\log\eps|)|\log\eps| + N (C_2''+C_2''')|\log\eps|\,.
\end{split}
\]
Comparing with Eq.~\eqref{E>} we deduce that
\begin{equation}
\label{G<}
\sum_i G_i^{(2)}(t) \le (C_1+NC_2') (\log|\log\eps|)|\log\eps| + N(C_2''+C_2''') |\log\eps|\,.
\end{equation}
But, by Eqs.~\eqref{A<} and \eqref{reps},
\[
G_i^{(2)}(t) \ge \frac12 (\alpha|\log\eps|-\bar d-\varrho)\,  \mc G_i(t)\,,
\]
and Eq.~\eqref{conc1} follows from Eq.~\eqref{G<} for a suitable choice of $C_2$.
\end{proof}

\subsection{Mass concentration and bound on the moment of inertia}
\label{sec:3.2}

As a consequence of Proposition \ref{prop:conc1}, we next prove that the mass of each vortex ring is concentrated in a disk of vanishing size as $\eps \to 0$.

\begin{theorem}
\label{thm:conc}
There exist constants $C_j=C_j(\alpha,|a|,\bar d, \varrho)>0$, $j=3,4$, and points $q^{i,\eps}(t)\in \bb R^2$, $t\in [0,T_\eps]$, $\eps \in (0,\eps_0)$, such that if $R>\exp(C_3 \log|\log\eps|)$ then, for any $\eps \in (0,\eps_0)$,
\begin{align}
\label{eq:conc}
\frac{a_i}{|a_i|} \int_{\Sigma(q^{i,\eps}(t)|\eps R)}\!\rmd x\, \omega_{i,\eps}(x,t) \ge |a_i| - \frac{C_4 \log|\log\eps|}{\log R} \quad \forall\, t\in [0,T_\eps] \quad \forall\, i = 1,\ldots,N \,.
\end{align}
\end{theorem}

\begin{proof}
In what follows we omit the explicit dependence on $i$ and $t$ by introducing the shortened notation
\[
\omega(x) := \frac{a_i}{|a_i|} \omega_{i,\eps}(x,t) =|\omega_{i,\eps}(x,t)| \,, \quad a= |a_i|\,.
\] 
Since $\int\!\rmd x\, \omega(x) = |a_i|$ we can find $x_1^*$ and $L_1>1$ such that
\[
M_1 := \int_{x_1< x_1^*-\eps L_1}\!\rmd x\,  \omega(x) \le \frac a2\,, \qquad  M_3 :=\int_{x_1> x_1^*+\eps L_1}\!\rmd x\,  \omega(x)\le \frac a2\,.
\] 
Setting
\[
M_2 := \int_{|x_1-x_1^*| \le \eps L_1}\!\rmd x\,  \omega(x)\,,
\]
from Eq.~\eqref{conc1}, by neglecting the vorticity in the central region, we deduce that $2M_1M_3 \log(2L_1) \le C_2 \log|\log\eps|$. Therefore,
\[
\begin{split}
a^2 & = (M_1+M_2+M_3)^2 \le \frac{a^2}2 + 2 M_2^2 + 2M_2(M_1+M_3) + 2 M_1 M_3 \\ & = \frac{a^2}2 + 2 a M_2 + 2 M_1M_3 \le \frac{a^2}2 + 2 a M_2 + C_2 \frac{\log|\log\eps|}{\log(2L_1)}\,,
\end{split}
\]
whence
\begin{equation}
\label{stco}
M_2 \ge \frac a4 - \frac{C_2 \log|\log\eps|}{2a \log(2L_1)}\,.
\end{equation}
In particular,
\begin{equation}
\label{a8}
M_2 \ge  \frac a8 \quad \forall \, L_1 \ge L_1^* := \frac 12 \exp\Big(\frac{4C_2}{a^2}\log|\log\eps|\Big)\,.
\end{equation}
Letting now
\[
M_1' := \int_{x_1< x_1^*- 2\eps L_1}\!\rmd x\,  \omega(x)\,, \quad  M_3' :=\int_{x_1> x_1^*+ 2\eps L_1}\!\rmd x\,  \omega(x)\,,
\]
from Eq.~\eqref{conc1}, neglecting some positive terms and using \eqref{a8}, we obtain
\[
\frac a8 (M_1'+M_3') \log L_1 \le  C_2 \log|\log\eps| \quad \forall \, L_1 \ge L_1^*\,,
\]
whence
\begin{equation}
\label{w>}
M_2' := \int_{|x_1-x_1^*| \le 2\eps L_1}\!\rmd x\,  \omega(x) \ge a -  \frac{8C_2\log|\log\eps|}{a\log L_1} \quad \forall \, L_1 \ge L_1^*\,.
\end{equation}
We can now repeat the same argument in the $x_2$-direction for the function
\[
\tilde \omega(x) = \omega(x) \id(|x_1-x_1^*| \le 2\eps L_1)\,,
\]
when $L_1>L_1^*$. It follows that there is $x_2^*$ such that
\begin{equation}
\label{tildw>}
\int_{|x_2-x_2^*| \le 2\eps L_2}\!\rmd x\,  \tilde \omega(x) \ge M_2' -  \frac{8C_2\log|\log\eps|}{M_2'\log L_2} \quad \forall \, L_2 \ge L_2^*\,,
\end{equation}
with now
\[
L_2^* := \frac 12 \exp\Big(\frac{4C_2}{(M_2')^2}\log|\log\eps|\Big) \le \frac 12 \exp\Big(\frac{256 C_2}{a^2}\log|\log\eps|\Big) \,,
\]
where in the last inequality we used that $M_2'\ge M_2 \ge a/8$ by Eq.~\eqref{a8}.

Therefore, letting $x^* = (x_1^*,x_2^*)$ and choosing
\[
L_1 = L_2 = L > \frac 12 \exp\Big(\frac{256 C_2}{a^2}\log|\log\eps|\Big),
\]
from Eqs.~\eqref{w>} and \eqref{tildw>} we get
\[
\begin{split}
\int_{\Sigma(x^*|2\eps \sqrt 2 L)} \! \rmd x\, \omega(x) & \ge \int_{|x_2-x_2^*| \le 2\eps L}\!\rmd x\,  \tilde \omega(x) \ge  a -  \frac{8C_2\log|\log\eps|}{a\log L} - \frac{8C_2\log|\log\eps|}{M_2'\log L}  \\ & \ge a - \frac{72 C_2\log|\log\eps|}{a\log L}\,,
\end{split}
\]
where we used again $M_2'\ge a/8$ in the last inequality. Coming back to the original notation, Eq.~\eqref{eq:conc} follows with $q^{i,\eps}(t) = x^*$ and suitable choices of $C_3,C_4>0$.
\end{proof}

We denote by $B^{i,\eps}(t)$ the center of vorticity of the $i$-th vortex ring, defined by
\begin{equation}
\label{c.m.}
B^{i,\eps}(t) := \frac{1}{a_i} \int\! \rmd x\, x\,\omega_{i,\eps}(x,t)\,, 
\end{equation}
and by $J_{i,\eps}(t)$ the corresponding moment of inertia, i.e.,
\begin{equation}
\label{J}
J_{i,\eps}(t) : = \int \!\rmd x\,  |x-B^{i,\eps}(t)|^2 |\omega_{i,\eps}(x,t)| = \frac{a_i}{|a_i|} \int \!\rmd x\,  |x-B^{i,\eps}(t)|^2 \omega_{i,\eps}(x,t) \,.
\end{equation}

From Theorem \ref{thm:conc}, we deduce that the moment of inertia vanishes as $\eps\to 0$. This is the content of the following theorem.

\begin{theorem}
\label{thm:J<}
Given $\gamma\in (0,1)$ there exists $\eps_\gamma \in (0,\eps_0)$ such that
\begin{equation}
\label{eq:J<}
J_{i,\eps}(t) \le \frac{1}{|\log\eps|^\gamma} \quad \forall\, t\in [0,T_\eps] \quad \forall\,\eps\in (0,\eps_\gamma)\,.
\end{equation}
\end{theorem}

\begin{proof}
Without loss of generality we assume hereafter $a_i>0$, i.e., $\omega_{i,\eps}(t) \ge 0$ so that 
\[
J_{i,\eps}(t) : = \int \!\rmd x\,  |x-B^{i,\eps}(t)|^2 \omega_{i,\eps}(x,t)\,.
\]
Given $\gamma\in (0,1)$, we choose $\gamma' \in (\gamma,1)$ and let
\begin{equation}
\label{eq:Sgt}
\Sigma_{i,\eps}(t) := \Sigma(q^{i,\eps}(t)|\eps R_\eps)\,, \quad R_\eps := \exp(|\log\eps|^{\gamma'})\,.
\end{equation}
By Theorem \ref{thm:conc}, provided $\eps\in (0,\eps_0)$ is chosen sufficiently small to have $R_\eps >\exp(C_3\log|\log\eps|)$, we can apply Eq.~\eqref{eq:conc} with $R=R_\eps$ getting
\begin{equation}
\label{stmass}
\int_{ \Sigma_{i,\eps}(t)}\!\rmd y\, \omega_\eps(y,t) \ge a_i - \frac{C_4 \log|\log\eps| }{|\log\eps|^{\gamma'}}  \quad \forall\, t\in [0,T_\eps]\,.
\end{equation}
Now, by the definition of center of vorticity,
\[
\begin{split}
J_{i,\eps}(t) & = \min_{q\in\bb R^2} \int\! \rmd x\,  |x-q|^2\omega_{i,\eps}(x,t)  \le  \int\! \rmd x\, |x-q^{i,\eps}(t)|^2\omega_{i,\eps}(x,t) \\ & = \int_{\Sigma_{i,\eps}(t)}\! \rmd x\, |x-q^{i,\eps}(t)|^2\omega_{i,\eps}(x,t) + \int_{\Sigma_{i,\eps}(t)^\complement}\! \rmd x\, |x-q^{i,\eps}(t)|^2\omega_{i,\eps}(x,t)  \\ & \le a_i (\eps R_\eps)^2 + \frac{C_4 \log|\log\eps| }{|\log\eps|^{\gamma'}}\max_{x\in\Lambda_{i,\eps}(t)} |x-q_\eps(t)|^2.
\end{split}
\]
Now, for $t\in [0,T_\eps]$ and $\eps$ small enough,
\[
\max_{x\in\Lambda_{i,\eps}(t)} |x-q^{i,\eps}(t)|^2 \le 2 \max_{x\in\Lambda_{i,\eps}(t)} |x|^2 + 2|q^{i,\eps}(t)|^2 \le 2(\bar d + \varrho)^2 + 2(\bar d +\varrho+\eps R_\eps)^2,
\]
where we used Eq.~\eqref{sep-supp} and that, in view of Eq.~\eqref{stmass}, $\Sigma_{i,\eps}(t) \cap\Lambda_{i,\eps}(t) \ne \emptyset$. Then, the lemma follows from the above estimates.
\end{proof}

\section{Iterative procedure}
\label{sec:4}

As already observed, our goal is to show that the time $T_\eps$ does not vanish as $\eps\to 0$. To this purpose, we will prove that there is $T_\varrho'\in (0,T]$ such that the condition on the supports $\Lambda_{i,\eps}(t)$ in the definition of $T_\eps$ can be enforced up to time $T_\varrho' \wedge T_\eps$ for any $\eps$ small enough. By continuity, this implies that $T_\eps\ge T_\varrho'$ for any $\eps$ small enough (and Theorem \ref{thm:1} will follow with $T_\varrho = T_\varrho'$).

A key step in this strategy, which is the content of the present section, is to prove that the amount of vorticity $\omega_{i,\eps}(x,t)$ outside any disk centered in $B^{i,\eps}(t)$ and whose radius is fixed independent of $\eps$ is indeed extremely small, say $\eps^\ell$ up to a time $\bar T_\ell \wedge T_\eps$.

To this end, a naive application of the iterative procedure adopted in the quoted papers fails in this case, because of the worst estimate Eq.~\eqref{eq:J<} on the moment of inertia. To overcome this difficulty, we notice that this estimate is sufficient to apply an iterative argument based on a larger space step, which gives a weaker estimate. But this estimate is good enough to implement a new iterative argument, now based on the correct smaller space step, which leads to the result.

The following lemma, whose proof is given in Appendix \ref{app:a}, will be repeatedly used in the sequel.
 
\begin{lemma}
\label{lem:Hdec}
Recall $H(x,y) = (H_1(x,y),H_2(x,y))$ is defined in Eqs.~\eqref{H1}, \eqref{H2}. There exists $C_H = C_H(\alpha,\bar d, \varrho) >0$ such that, for any $i\ne j$, $t \in [0,T]$, and $\eps\in (0,\eps_0)$,
\begin{equation}
\label{stimH}
|H(x,y)| + \|D_xH(x,y)\| \le C_H\quad \forall\, x\in\Sigma(\zeta^i(t)|\varrho) \quad\forall\, y \in\Sigma(\zeta^j(t)|\varrho)
\end{equation}
($D_xH(x,y)$ denotes the Jacobian matrix of $H(x,y)$ with respect to the variable $x$). Moreover, $H(x,y)$ admits the decomposition,
\begin{equation}
\label{sH}
H(x,y) = K(x-y) + L(x,y) + \mc R(x,y), 
\end{equation}
where $K(\cdot)$ is defined in \eqref{Kern},
\begin{equation}
\label{sL}
L(x,y) = \frac{1}{4\pi (r_\eps+x_2)} \log\frac {1+|x-y|}{|x-y|}\begin{pmatrix} 1 \\ 0 \end{pmatrix},
\end{equation}
and, for a suitable constant $\bar C>0$,
\begin{equation}
\label{sR}
|\mc R(x,y)| \le \bar C \frac{1+ r_\eps + x_2+\sqrt{A} \big(1+ |\log A|\big)}{(r_\eps + x_2)^2}\,,
\end{equation}
with $A$ as in Eq.~\eqref{Ab}.
\end{lemma}

We decompose the velocity field Eq.~\eqref{u=} according to  Eq.~\eqref{sH}, writing
\begin{equation}
\label{decom_u}
u(x,t) = (K*\omega_{i,\eps})(x,t) + u^i_L(x,t) + u^i_{\mc R}(x,t) + F^i(x,t)\,,
\end{equation}
where $(K*\omega_{i,\eps})(\cdot, t)$ denotes the convolution of $K$ and $\omega_{i,\eps}(\cdot,t)$,
\begin{equation}
\label{uL=w}
u^i_L(x,t) := \int\!\rmd y\, L(x,y)\, \omega_{i,\eps}(y,t) = w^i_L(x,t) \begin{pmatrix} 1 \\ 0 \end{pmatrix},
\end{equation}
with
\begin{equation}
\label{wL=}
w^i_L(x,t) := \frac{1}{4\pi (r_\eps + x_2)} \int\!\rmd y\, \omega_{i,\eps}(y,t)\log\frac {1+|x-y|}{|x-y|}\,,
\end{equation}
and 
\[ \quad u^i_{\mc R}(x,t) := \int\!\rmd y\, \mc R(x,y)\, \omega_{i,\eps}(y,t)\,,\quad  F^i(x,t) := \sum_{j\ne i}  \int\!\rmd y\, H(x,y)\, \omega_{j,\eps}(y,t)\,.
\]

By Eq.~\eqref{sep-disks} and in view of Eqs.~\eqref{stimH} and \eqref{sR}, for any $\eps\in (0,\eps_0)$,
\begin{equation}
\label{stimF}
|F^i(x,t)| + \|D_xF^i(x,t)\| \le C_F \quad \forall\, x\in \Sigma(\zeta^i(t)|\varrho) \quad  \forall\, t \in [0,T_\eps]\,,
\end{equation}
with $C_F := |a| C_H$, and there is $C_{\mc R} = C_{\mc R}(\alpha,|a|,\bar d, \varrho) >0$ such that
\begin{equation}
\label{stimR}
|u^i_{\mc R}(x,t)| \le \frac{C_{\mc R}\log|\log\eps|}{|\log\eps|} \quad \forall\, x\in \bigcup_j \Sigma(\zeta^j(t)|\varrho) \quad  \forall\, t \in [0,T_\eps]\,.
\end{equation}
Moreover, since $s \mapsto\log[(1+s)/s]$ is decreasing, the integral appearing in the definition of $w^i_L(x,t)$ can be estimated performing a symmetrical rearrangement of the vorticity around the point $x$. Therefore, recalling Eqs.~\eqref{sep-disks} and \eqref{omt<}, if $(x,t) \in \Sigma(\zeta^i(t)|\varrho) \times [0,T_\eps]$ then
\begin{align*}
|w^i_L(x,t)| & \le \frac{1}{4\pi(r_\eps-\bar d-\varrho)} \int\!\rmd y\, \log\frac {1+|x-y|}{|x-y|} |\omega_{i,\eps}(y,t)| \\ & \le \frac{1}{4\pi(r_\eps-\bar d-\varrho)} \frac{r_\eps+\bar d+\varrho}{r_\eps-\eps} \frac{M}{\eps^2}  \int_0^{\bar r}\!\rmd r\, 2\pi r \, \log\frac{1+r}{r} \\ & =  \frac{M(r_\eps+\bar d+\varrho)}{2(r_\eps-\bar d-\varrho)(r_\eps-\eps)\eps^2} \bigg\{\frac{\bar r^2}{2} \log\frac{1+\bar r}{\bar r} + \frac 12 \int_0^{\bar r}\!\rmd r\, \frac{r}{1+r} \bigg\},
\end{align*}
with $\bar r$ as in Eq.~\eqref{br=}. Therefore there is $C_L = C_L(\alpha,|a|,\bar d, \varrho)>0$ such that, for any $\eps\in (0,\eps_0)$,
\begin{equation}
\label{uL<}
|u^i_L(x,t)| = |w^i_L(x,t)| \le \frac{|a_i|}{4\pi\alpha} + \frac{C_L}{|\log\eps|} \quad \forall\, x\in \Sigma(\zeta^i(t)|\varrho) \quad \forall\, t\in [0,T_\eps]\,.
\end{equation}

\begin{proposition}
\label{prop:1}
Let
\[
m^i_t(R) := \int_{\Sigma(B^{i,\eps}(t)|R)^\complement}\!\rmd x\, |\omega_{i,\eps}(x,t)| =  \frac{a_i}{|a_i|} \int_{\Sigma(B^{i,\eps}(t)|R)^\complement}\!\rmd x\, \omega_{i,\eps}(x,t)
\]
denote the amount of vorticity of the $i$-th ring outside the disk $\Sigma(B^{i,\eps}(t)|R)$ at time $t$. Then, given $R>0$, for each $\ell>0$ there is $\widetilde T_\ell \in (0, T]$ such that 
\begin{equation}
\label{smt1}
m^i_t(R) \le \frac{1}{|\log\eps|^{\ell}} \quad \forall\, t\in[0, \widetilde T_\ell  \wedge T_\eps] \quad \forall\, i = 1,\ldots,N \,.
\end{equation}
for any $\eps\in (0,\eps_0)$ sufficiently small.
\end{proposition}

\begin{proof}
Without loss of generality we assume hereafter $a_i=|a_i|$, i.e., $\omega_{i,\eps}(t) \ge 0$, so that 
\[
m^i_t(R) := \int_{\Sigma(B^{i,\eps}(t)|R)^\complement}\!\rmd x\, \omega_{i,\eps}(x,t)\,.
\]
In what follows, $h$ and $R$ are two positive parameters to be fixed later and such that $R\ge 2h$. Let $x\mapsto W_{R,h}(x)$, $x\in \bb R^2$, be a non-negative smooth function, depending only on $|x|$, such that
\begin{equation}
\label{W1}
W_{R,h}(x) = \begin{cases} 1 & \text{if $|x|\le R$}, \\ 0 & \text{if $|x|\ge R+h$}, \end{cases}
\end{equation}
and, for some $C_W>0$,
\begin{equation}
\label{W2}
|\nabla W_{R,h}(x)| < \frac{C_W}{h}\,,
\end{equation}
\begin{equation}
\label{W3}
|\nabla W_{R,h}(x)-\nabla W_{R,h}(x')| < \frac{C_W}{h^2}\,|x-x'|\,. 
\end{equation}
The quantity
\begin{equation}
\label{mass 1}
\mu_t(R,h) = \int\! \rmd x \, \big[1-W_{R,h}(x-B^{i,\eps}(t))\big]\, \omega_{i,\eps}(x,t)\,,
\end{equation}
is a mollified version of $m^i_t$, satisfying
\begin{equation}
\label{2mass 3}
\mu_t(R,h) \le m^i_t(R) \le \mu_t(R-h,h)\,,
\end{equation}
so that it is enough to prove \eqref{smt1} with $\mu_t$ in place of $m^i_t$. To this end, we study the variation in time of $\mu_t(R,h)$ by applying \eqref{weq} with test function $f(x,t) = 1-W_{R,h}(x-B^{i,\eps}(t))$, getting
\[
\frac{\rmd}{\rmd t} \mu_t(R,h) = - \int\! \rmd x\, \nabla W_{R,h}(x-B^{i,\eps}(t)) \cdot [u(x,t) - \dot B^{i,\eps}(t)]\,\omega_{i,\eps}(x,t)\,.
\]
The time derivative of the center of vorticity can be computed by applying again Eq.~\eqref{weq} (with now test function $f(x,t) =x$), so that
\begin{align}
\label{bpunto}
\dot B^{i,\eps}(t) = \frac{1}{a_i} \int\!\rmd x\, \big[ u^i_L(x,t) + u^i_{\mc R}(x,t) + F^i(x,t)\big] \omega_{i,\eps}(x,t)\,,
\end{align}
having used the decomposition Eq.~\eqref{decom_u} and that $\int\!\rmd x\, \omega_{i,\eps}(x,t) (K*\omega_{i,\eps})(x,t) = 0$ by the antisymmetry of $K$. We thus conclude that
\begin{equation}
\label{mass 4}
\frac{\rmd}{\rmd t} \mu_t(R,h) =  - A_1 - A_2 - A_3\,,
\end{equation}
with
\begin{equation*}
\begin{split}
A_1 & = \int\! \rmd x\, \nabla W_{R,h}(x-B^{i,\eps}(t)) \cdot (K*\omega_{i,\eps})(x,t)\,  \omega_{i,\eps}(x,t)  \\ & = \frac 12 \int\! \rmd x \! \int\! \rmd y\, [\nabla W_{R,h}(x-B^{i,\eps}(t)) - \nabla W_{R,h}(y-B^{i,\eps}(t))] \\ & \qquad  \cdot K(x-y) \, \omega_{i,\eps}(x,t)\,  \omega_{i,\eps}(y,t) \\ A_2 & = \int\! \rmd x\, \nabla W_{R,h}(x-B^{i,\eps}(t)) \cdot \big[u^i_L(x,t) + u^i_{\mc R} (x,t) \big] \omega_{i,\eps}(x,t) \\ & \;\; - \frac{1}{a_i}\int\! \rmd x\, \nabla W_{R,h}(x-B^{i,\eps}(t)) \cdot \int\! \rmd y \,\big[u^i_L(y,t) + u^i_{\mc R} (y,t) \big]  \omega_{i,\eps}(y,t)\, \omega_{i,\eps}(x,t)\,, \\ A_3 & = \frac{1}{a_i} \int\! \rmd x\, \nabla W_{R,h}(x-B^{i,\eps}(t)) \cdot \int\! \rmd y \,\big[F^i(x,t) - F^i (y,t) \big]  \omega_{i,\eps}(y,t)\, \omega_{i,\eps}(x,t)\,,
\end{split}
\end{equation*}
where the second expression of $A_1$ is due to the antisymmetry of $K$.

Concerning $A_1$, we introduce the new variables $x'=x-B^{i,\eps}(t)$, $y'=y-B^{i,\eps}(t)$, define $\tilde\omega_{i,\eps}(z,t) := \omega_{i,\eps}(z+B^{i,\eps}(t),t)$, and let
\[
f(x',y') = \frac 12 \tilde\omega_{i,\eps}(x',t)\, \tilde\omega_{i,\eps}(y',t) \, [\nabla W_{R,h}(x')-\nabla W_{R,h}(y')] \cdot K(x'-y') \,,
\]
so that $A_1 = \int\!\rmd x' \! \int\!\rmd y'\,f(x',y')$. We observe that $f(x',y')$ is a symmetric function of $x'$ and $y'$ and that, by \eqref{W1}, a necessary condition to be different from zero is if either $|x'|\ge R$ or $|y'|\ge R$. Therefore,
\begin{equation*}
\begin{split}
A_1  &= \bigg[ \int_{|x'| > R}\!\rmd x' \! \int\!\rmd y' + \int\!\rmd x' \! \int_{|y'| > R}\!\rmd y' -  \int_{|x'| > R}\!\rmd x' \! \int_{|y'| > R}\!\rmd y'\bigg]f(x',y') \\ & = 2 \int_{|x'| > R}\!\rmd x' \! \int\!\rmd y'\,f(x',y')  -  \int_{|x'| > R}\!\rmd x' \! \int_{|y'| > R}\!\rmd y'\,f(x',y') \\ & = A_1' + A_1'' + A_1'''\,,
\end{split}
\end{equation*}
with 
\begin{equation*}
\begin{split}
A_1' & = 2 \int_{|x'| > R}\!\rmd x' \! \int_{|y'| \le \frac{R}{2}}\!\rmd y'\,f(x',y') \,, \quad A_1'' = 2 \int_{|x'| > R}\!\rmd x' \! \int_{|y'| > \frac{R}{2}}\!\rmd y'\,f(x',y')\,, \\ A_1''' & = -  \int_{|x'| > R}\!\rmd x' \! \int_{|y'| > R}\!\rmd y'\,f(x',y')\,.
\end{split}
\end{equation*}
By the assumptions on $W_{R,h}$, we have $\nabla W_{R,h}(z) = \eta_h(|z|) z/|z|$ with $\eta_h(|z|) =0$ for $|z| \le R$. In particular, $\nabla W_{R,h}(y') = 0$ for $|y'| \le R/2$, hence
\[
A_1' =  \int_{|x'| > R}\!\rmd x' \, \tilde\omega_{i,\eps}(x',t) \eta_h(|x'|) \,\frac{x'}{|x'|} \cdot  \int_{|y'| \le \frac{R}{2}}\!\rmd y'\, K(x'-y') \, \tilde\omega_{i,\eps}(y',t)\,.
\]
In view of  \eqref{W2}, $|\eta_h(|z|)| \le C_W/h$, so that 
\begin{equation}
\label{a1'}
|A_1'| \le \frac{C_W}{h} m^i_t(R) \sup_{|x'| > R} |H_1(x')|\,,
\end{equation}
with
\[
H_1(x') = \frac{x'}{|x'|}\cdot  \int_{|y'| \le \frac{R}{2}}\!\rmd y'\, K(x'-y') \, \tilde\omega_{i,\eps}(y',t) \,.
\]
Now, recalling \eqref{Kern} and using that $x'\cdot (x'-y')^\perp=-x'\cdot y'^\perp$, we get,
\begin{equation}
\label{in H_11}
H_1(x') = \frac{1}{2\pi} \int_{|y'|\leq \frac{R}{2}}\! \rmd y'\, \frac{x'\cdot y'^\perp}{|x'||x'-y'|^2}\,  \tilde\omega_{i,\eps}(y',t) \,.
\end{equation}
By \eqref{c.m.}, $\int\! \rmd y'\,  y'^\perp\,  \tilde\omega_{i,\eps}(y',t) = 0$, so that
\begin{equation}
\label{in H_13}
H_1(x')  = H_1'(x')-H_1''(x')\,, 
\end{equation}
where
\begin{eqnarray*}
&& H_1'(x') = \frac{1}{2\pi}  \int_{|y'|\le \frac{R}{2}}\! \rmd y'\, \frac {x'\cdot y'^\perp}{|x'|}\, \frac {y'\cdot (2x'-y')}{|x'-y'|^2 \ |x'|^2} \,  \tilde\omega_{i,\eps}(y',t) \,, \\ && H_1''(x')= \frac{1}{2\pi} \int_{|y'|> \frac{R}{2}}\! \rmd y'\, \frac{x'\cdot y'^\perp}{|x'|^3}\,  \tilde\omega_{i,\eps}(y',t) \,.
\end{eqnarray*}
We notice that if $|x'| > R$ then $|y'| \le \frac{R}{2}$ implies $|x'-y'|\ge \frac{R}{2}$ and $|2x'-y'|\le |x'-y'|+|x'|$. Therefore, for any $|x'| > R$,
\[
|H_1'(x')| \le \frac 1\pi \bigg[\frac{1}{|x'|^2 R} + \frac{2}{|x'|R^2} \bigg]  \int_{|y'|\leq \frac{R}{2}} \! \rmd y'\, |y'|^2 \,  \tilde\omega_{i,\eps}(y',t) \le \frac{3 J_{i,\eps}(t)}{\pi R^3}\,.
\]
To bound $H_1''(x')$, by Chebyshev's inequality, for any $|x'| > R$ we have,
\[
|H_1''(x')| \le \frac{1}{2\pi |x'|^2} \int_{|y'|> \frac{R}{2}}\! \rmd y'\, |y'| \tilde\omega_{i,\eps}(y',t) \le \frac{J_{i,\eps}(t)}{ \pi R^3} \,.
\]
From Eqs.~\eqref{a1'}, \eqref{in H_13}, and the previous estimates, we conclude that
\begin{equation}
\label{H_14b}
|A_1'| \le \frac{4C_W J_{i,\eps}(t)}{\pi R^3  h} m^i_t(R)\,.
\end{equation}
Now, by \eqref{W3} and then applying the Chebyshev's inequality,
\begin{align}
\label{badterm}
|A_1''| + |A_1'''| & \le \frac{C_W}{\pi h^2} \int_{|x'| \ge R}\!\rmd x' \! \int_{|y'| \ge \frac{R}{2}}\!\rmd y'\,\tilde\omega_{i,\eps}(y',t) \,  \tilde\omega_{i,\eps}(x',t) \nonumber \\ & = \frac{C_W}{\pi h^2}m^i_t(R)   \int_{|y'| \ge \frac{R}{2}}\!\rmd y'\, \tilde\omega_{i,\eps}(y',t)  \le \frac{4C_W J_{i,\eps}(t)}{\pi R^2h^2} m^i_t(R)\,.
\end{align}
In conclusion, 
\begin{equation}
\label{a1s}
|A_1| \le   \frac{4C_W}{\pi} \left( \frac{1}{R^3 h} +  \frac{1}{R^2 h^2}\right) J_{i,\eps}(t) m^i_t(R)\,.
\end{equation}

Concerning $A_2$ and $A_3$, we observe that by \eqref{W1} the integrand is different from zero only if $R\le |x-B^{i,\eps}(t)|\le R+h$ and $x,y\in \Lambda_{i,\eps}(t) \subset \Sigma(\zeta^i(t)|\varrho)$. Therefore, by Eqs.~\eqref{stimR}, \eqref{uL<}, using again the variables $x'=x-B^{i,\eps}(t)$, $y'=y-B^{i,\eps}(t)$, and that $\int\!\rmd y\, \omega_{i,\eps}(y,t)=a_i$, see Eq.~\eqref{ait},
\begin{equation}
\label{a2s}
|A_2| \le \frac{2C_W}{h}\left[\frac{|a_i|}{4\pi\alpha} + \frac{C_{\mc R}\log|\log\eps|+C_L}{|\log\eps|} \right] m^i_t(R)\,,
\end{equation}
while, from the bounds on $F^i$ and its Lipschitz constant in Eq.~\eqref{stimF}, 
\begin{align}
\label{a3s}
|A_3| & \le \frac{2C_W C_F}{a_ih} \int_{|x'|\ge R}\!\rmd x'  \tilde\omega_{i,\eps}(x',t)  \int_{|y'|> R}\!\rmd y'\, \tilde\omega_{i,\eps}(y',t) \nonumber \\ & \quad + \frac{C_WC_F}{a_ih} \int_{R \le |x'|\le R+h}\!\rmd x'  \tilde\omega_{i,\eps}(x',t) \int_{|y'| \le R}\!\rmd y'\,|x'- y'| \,  \tilde\omega_{i,\eps}(y',t) \nonumber \\ & \le \frac{2C_W C_F J_{i,\eps}(t)}{a_iR^2h} m^i_t(R) +  C_WC_F \bigg(1+\frac{2R}h\bigg) m^i_t(R)\,,
\end{align}
where we used that $|x'-y'| \le 2R+h$ in the domain of integration of the last integral and the Chebyshev's inequality in the first one.

From Eqs.~\eqref{a1s}, \eqref{a2s}, \eqref{a3s}, and Theorem \ref{thm:J<} we deduce that
\begin{equation}
\label{2mass 4''}
\frac{\rmd}{\rmd t} \mu_t(R,h) \le A_\eps(R,h) m^i_t(h)\quad \forall\, t\in [0,T_\eps] \,,
\end{equation}
where, for each $\gamma \in (0,1)$,
\begin{align}
\label{mass 4bis}
A_\eps(R,h) &= \frac{2C_W}{h} \bigg[ C_FR + C_F \frac h2 + \frac{|a_i|}{4\pi\alpha} + \frac{C_{\mc R}\log|\log\eps|+C_L}{|\log\eps|} + \frac{C_F}{a_i|\log\eps|^{\gamma}R^2}  \nonumber \\ & \qquad\qquad +  \frac{1}{|\log\eps|^{\gamma}R^3} + \frac{1}{|\log\eps|^{ \gamma}R^2  h} \bigg],
\end{align}
for any $\eps \in (0,\eps_\gamma)$ with $\eps_\gamma$ as in Theorem \ref{thm:J<}. Therefore, by Eqs.~\eqref{2mass 3} and \eqref{2mass 4''},
\begin{equation}
\label{mass 14'}
\mu_t(R,h) \le \mu_0(R,h) + A_\eps(R,h) \int_{0}^t \rmd s\, \mu_s(R-h,h) \quad \forall\, t\in [0,T_\eps] \,.
\end{equation}
We iterate the last inequality $n = \lfloor\log|\log\eps|\rfloor$ times,\footnote{$\lfloor z\rfloor$ denotes the integer part of the positive number $z$.} from $R_0 = R - h$ to $R_n = R -(n+1)h = R/2$. Since $h = R/(2n+2)$ and $R_j\in [R/2, R]$, from the explicit expression Eq.~\eqref{mass 4bis} and using that $|a_i|  < |a|$, it is readily seen that if $\eps$ is sufficiently small then
\[
A_\eps(R_j,h) \le A_* \frac nR \quad \forall\, j=0,\ldots,n\,,
\]
with
\begin{equation}
\label{a*}
A_*= 4 C_W \left(C_FR +\frac{|a|}{4\pi\alpha}\right).
\end{equation}
Therefore, for any $\eps$ small enough and $t\in [0,T_\eps]$,
\[
\begin{split}
\mu_t(R-h,h) & \le \mu_0(R-h,h) + \sum_{j=1}^n \mu_0(R_j,h) \frac{(A_*nt/R)^j}{j!} \\ & \quad + \frac{(A_*n/R)^{n+1}}{n!} \int_0^t\!\rmd s\,  (t-s)^n\mu_s(R_{n+1},h) \,.
\end{split}
\]
Since $\Lambda_{i,\eps}(0) \subset \Sigma(\zeta^i_0|\eps)$, if $\eps$ is  sufficiently small then $\mu_0(R_j,h)=0$ for any $j=0,\ldots,n$. Therefore, recalling also Eq.~\eqref{2mass 3}, for any $\eps$ small enough,
\begin{align}
\label{mass 15'}
m^i_t(R) & \le \mu_t(R-h,h) \le \frac{(A_*n/R)^{n+1}}{n!} \int_0^t\!\rmd s\,  (t-s)^n\mu_s(R_{n+1},h) \nonumber \\ &  \le \frac{9}{R^2|\log\eps|^\gamma}  \frac{(A_*nt/R)^{n+1}}{(n+1)!} \le \frac{9}{R^2|\log\eps|^\gamma}  \left(\frac{\rme A_* t}{R}\right)^{n+1} \quad \forall\, t\in [0,T_\eps]\,,
\end{align}
where we used the Chebyshev's inequality and Theorem \ref{thm:J<} in the third inequality, to estimate
\[
\mu_s(R_{n+1},h) \le m^i_s(R_{n+1}) = m^i_s(R/2) \le \frac{J_{i,\eps}(s)}{(R/2-h)^2} \le \frac{9}{R^2|\log\eps|^\gamma}\,, 
\]
and the Stirling approximation in the last one. Since $n = \lfloor\log|\log\eps|\rfloor$ Eq.~\eqref{mass 15'} implies the bound \eqref{smt1} for any $\eps$ sufficiently small choosing, e.g., $\widetilde T_\ell = (R/A_*)\rme^{-\ell-1} \wedge T$.
\end{proof}

\begin{proposition}
\label{prop:2}
Let $m^i_t(R)$ be as in Proposition \ref{prop:1}. Then, given $R>0$, for each $\ell>0$ there is $\bar T_\ell \in (0, T]$ such that 
\begin{equation}
\label{smt}
m^i_t(R) \le \eps^{\ell} \quad \forall\, t\in[0, \bar T_\ell  \wedge T_\eps]
\end{equation}
for any $\eps\in (0,\eps_0)$ sufficiently small.
\end{proposition}

\begin{proof}
The strategy used in the proof of Proposition \ref{prop:1} would give the stronger estimate Eq.~\eqref{smt} if we could choose $n = \lfloor |\log\eps|\rfloor$. But this means $h  \sim |\log\eps|^{-1}$, which seems not acceptable since it implies that the last term in the right-hand side of Eq.~\eqref{mass 4bis} diverges as $\eps$ vanishes. This dangerous term comes from Eq.~\eqref{badterm}, where the term $ \int_{|y'| \ge \frac{R}{2}}\!\rmd y'\, \tilde\omega_{i,\eps}(y',t) = m^i_t(R/2)$ is bounded from above by the moment of inertia. But now, Proposition \ref{prop:1} applied with $R/4$ in place of $R$ and, e.g., $\ell=2$, gives
\[
m^i_t(R/4) \le \frac{1}{|\log\eps|^2} \quad \forall\,  t\in[0, \widetilde T_2  \wedge T_\eps]
\]
for any $\eps$ small enough. Therefore, besides the bound Eq.~\eqref{badterm} (which holds for any $t\in [0,T_\eps]$), we also have
\[
|A_1''| + |A_1'''| \le \frac{C_W}{\pi |\log\eps|^2 h^2} m^i_t(R) \quad  \forall\,  t\in[0, \widetilde T_2  \wedge T_\eps]\,.
\]
We thus arrive, in place of Eqs.~\eqref{mass 14'}, to the integral inequality
\[
\mu_t(R',h) \le \mu_0(R',h) + A_\eps(R',h) \int_{0}^t \rmd s\, \mu_s(R'-h,h) \quad \forall\, t\in [0,\widetilde T_2  \wedge T_\eps] \,,
\]
with now
\[
\begin{split}
A_\eps(R',h) &= \frac{2C_W}{h} \bigg[ C_FR' + C_F\frac h2 + \frac{|a_i|}{4\pi\alpha} + \frac{C_{\mc R}\log|\log\eps|+C_L}{|\log\eps|} + \frac{C_F}{a_i|\log\eps|^{\gamma}(R')^2}  \nonumber \\ & \qquad\qquad +  \frac{1}{|\log\eps|^{\gamma}(R')^3} + \frac{1}{|\log\eps|^2  h} \bigg]
\end{split}
\]
for any $R'\ge R/2$ and $\eps$ small enough. This inequality can be iterated  $n = \lfloor|\log\eps|\rfloor$ times, from $R'_0 = R-h$ to $R'_n = R -(n+1)h = R/2$, and arguing as done in Proposition \ref{prop:1} we obtain, for any $\eps$ small enough,
\[
m^i_t(R) \le\frac{9}{R^2|\log\eps|^\gamma} \left(\frac{\rme A_* t}{R}\right)^{n+1} \quad \forall\, t\in [0,\widetilde T_2 \wedge T_\eps] \,,
\]
which implies the bound Eq.~\eqref{smt} for any $\eps$ sufficiently small choosing $\bar T_\ell = (R/A_*) \rme^{-\ell-1} \wedge \widetilde T_2$.
\end{proof}

\begin{remark}
\label{rem:4.1}
For later discussion, we give an explicit lower bound of the threshold $\bar T_\ell$ when $\ell>2$. From the proof of Proposition \ref{prop:1}, $\widetilde T_\ell = (R/A_*) \rme^{-\ell-1} \wedge T$, $A_*$ as in Eq.~\eqref{a*}, that with $R/4$ in place of $R$ and $\ell=2$ gives
\[
\widetilde T_2 =  \frac{R\rme^{-3}}{4C_W} \left(C_F R +\frac{|a|}{\pi\alpha}\right)^{-1} \wedge T.
\]
Therefore Eq.~\eqref{smt} holds for any $i=1,\ldots, N$, choosing, e.g.,
\[
\bar T_\ell = \frac{R\rme^{- \ell -1}}{4C_W} \left(C_F R +\frac{|a|}{\pi\alpha}\right)^{-1} \wedge T.
\]
\end{remark}

\section{Localization of vortices support}
\label{sec:5}

To enforce the condition on the support of the vortex rings in Eq.~\eqref{Teps}, we first show that these supports remain confined inside small disks centered in the corresponding centers of vorticity. To  this end, we need to evaluate the force acting on the fluid particles furthest from the center of vorticity.

\begin{lemma}
\label{lem:5.1}
Recall the definition Eq.~\eqref{Teps} of $T_\eps$ and define
\begin{equation}
\label{Rt}
R_t:= \max\{|x-B^{i,\eps}(t)|\colon x\in \Lambda_{i,\eps}(t)\}.
\end{equation}
Let $x(t)$ be the solution to \eqref{eqchar_n} with initial condition $x(0) = x_0 \in \Lambda_{i,\eps}(0)$ and suppose at time $t\in (0,T_\eps)$ it happens that 
\begin{equation}
\label{hstimv}
|x(t)-B^{i,\eps}(t)| = R_t.
\end{equation}
Then, at this time $t$, for each fixed $\gamma\in(0,1)$ and any $\eps$ small enough,
\begin{equation}
\label{stimv}
\frac{\rmd}{\rmd t} |x(t)- B^{i,\eps}(t)| \le 2 C_F R_t + \frac{|a_i|}{\pi\alpha} +  \frac{4}{\pi|\log\eps|^\gamma R_t^3} + \sqrt{\frac{M m^i_t(R_t/2)}{\eps^2}}\,,
\end{equation}
with $M$ as in Eq.~\eqref{Mgamma} and $C_F$ as in Eq~\eqref{stimF}.
\end{lemma}

\begin{proof}
Letting $x=x(t)$, from \eqref{decom_u} and \eqref{bpunto} we have,
\[
\begin{split}
\frac{\rmd}{\rmd t} |x- B^{i,\eps}(t)| & = \big(u(x,t) - \dot B^{i,\eps}(t)\big) \cdot \frac{x-B^{i,\eps}(t)}{|x-B^{i,\eps}(t)|} \\ & = (K*\omega_{i,\eps})(x,t) \cdot  \frac{x-B^{i,\eps}(t)}{|x-B^{i,\eps}(t)|} + U(x,t)\,,
\end{split}
\]
with
\[
\begin{split}
U(x,t) & = \bigg[ u^i_L(x,t) + u^i_{\mc R}(x,t)  -  \frac{1}{a_i} \int\!\rmd y\, \, \big( u^i_L(y,t) + u^i_{\mc R}(y,t)\big)  \omega_{i,\eps}(y,t) \\ & \quad + \frac{1}{a_i} \int\!\rmd y\, \, \big[F^i(x,t) - F^i(y,t)\big] \omega_{i,\eps}(y,t) \bigg] \cdot \frac{x-B^{i,\eps}(t)}{|x-B^{i,\eps}(t)|}\,.
\end{split}
\]
To evaluate the first term in the right-hand side, we split the domain of integration into the disk $\mc D=\Sigma(B^{i,\eps}(t)|R_t/2)$ and the annulus $\mc A= \Sigma(B^{i,\eps}(t)|R_t)\setminus\Sigma(B^{i,\eps}(t)|R_t/2)$. Then,
\begin{equation}
\label{in A_1,A_2}
(K*\omega_{i,\eps})(x,t) \cdot \frac{x-B^{i,\eps}(t)}{|x-B^{i,\eps}(t)|} = H_\mc D + H_\mc A,
\end{equation}
where
\[
H_\mc D(x) =  \frac{x-B^{i,\eps}(t)}{|x-B^{i,\eps}(t)|} \cdot \int_{\mc D}\! \rmd y\, K(x-y)\, \omega_{i,\eps}(y,t) 
\]
and
\[
H_\mc A(x)=  \frac{x-B^{i,\eps}(t)}{|x-B^{i,\eps}(t)|} \cdot \int_{\mc A}\! \rmd y\, K(x-y)\, \omega_{i,\eps}(y,t).
\]

We observe that $H_\mc D(x)$ is exactly equal to the integral $H_1(x')$ appearing in Eq.~\eqref{a1'}, provided $x'=x-B^{i,\eps}(t)$ and $R=R_t$. Moreover, to obtain Eq.~\eqref{H_14b} we had to bound $H_1(x')$ for $|x'|\ge R$, which is exactly what we need now, as $|x-B^{i,\eps}(t)|=R_t$.  This estimate, adapted to the present context becomes
\begin{equation}
\label{H_14}
|H_\mc D| \le \frac{4 J_{i,\eps}(t)}{\pi R_t^3} \le \frac{4}{\pi|\log\eps|^\gamma R_t^3}\,,
\end{equation}
where the last inequality holds for given $\gamma\in (0,1)$ and any $\eps$ small enough according to Eq.~\eqref{eq:J<}. Regarding $H_\mc A$, by the definition Eq.~\eqref{Kern} we have,
\begin{equation*}
|H_\mc A| \le \frac{1}{2\pi} \int_{\mc A}\! \rmd y\, \frac 1{|x-y|} \, |\omega_{i,\eps}(y,t)|\,.
\end{equation*}
Since the integrand is monotonically unbounded as $y\to x$, the maximum possible value of the integral can be obtained performing a symmetrical rearrangement of the vorticity around the point $x$. In view of Eq.~\eqref{omt<} and since $m^i_t(R_t/2)$ is equal  to the total amount of vorticity in $\mc A$, this rearrangement reads,
\[
|H_\mc A| \le \frac{r_\eps+\bar d+\varrho}{r_\eps-\eps} \frac{M}{2\pi\eps^2} \int_{\Sigma (0|\rho_*)}\!\rmd y'\, \frac{1}{|y'|} =\frac{r_\eps+\bar d+\varrho}{r_\eps-\eps} \frac{M}{\eps^2}  \rho_*, 
\]
where the radius $\rho_*$ is such that 
\[
\frac{r_\eps+\bar d+\varrho}{r_\eps-\eps} \frac{M}{\eps^2}\pi\rho_*^2 = m^i_t(R_t/2)\,.
\]
Therefore,
\begin{equation}
\label{h2}
|H_\mc A| \le \sqrt{\frac{r_\eps+\bar d+\varrho}{r_\eps-\eps} \frac{Mm^i_t(R_t/2)}{\pi\eps^2}} \le  \sqrt{\frac{M m^i_t(R_t/2)}{\eps^2}},
\end{equation}
where the second inequality is valid for $\eps$ small enough. Finally, by Eqs.~\eqref{stimR}, \eqref{uL<} (recall $x=x(t) \in \Lambda_{i,\eps}(t) \subset \Sigma(\zeta^i(t)|\varrho)$) and the bound on the Lipschitz constant of $F^i$ in Eq.~\eqref{stimF},
\[
|U(x,t)|  \le 2 \bigg[C_FR_t + \frac{|a_i|}{4\pi\alpha} + \frac{C_{\mc R}\log|\log\eps|+C_L}{|\log\eps|}\bigg].
\]
From this, Eqs.~\eqref{H_14} and \eqref{h2}, the estimate Eq.~\eqref{stimv} follows provided $\eps$ is chosen sufficiently small.
\end{proof}

\begin{proposition}
\label{prop:5.1}
There exists $T_\varrho' \in (0,T]$ such that, for any $\eps$ small enough,
\begin{equation}
\label{supp_pro}
\Lambda_{i,\eps}(t) \subset \Sigma(B^{i,\eps}(t)|\varrho/2) \quad \forall\, t\in [0, T_\varrho'\wedge T_\eps]\,.
\end{equation}
\end{proposition}

\begin{proof}
Let $\bar T_3$ be as in Proposition \ref{prop:2} with the choice $R=\varrho/10$ and $\ell=3$. Recalling the definition Eq.~\eqref{Rt}, we set
\[
t_1 := \sup\{t\in [0,\bar T_3 \wedge T_\eps] \colon R_s \le \varrho/2 \;\, \forall s \in [0,t]\}\,,
\]
If $t_1 = \bar T_3 \wedge T_\eps$, Eq.~\eqref{supp_pro} is proved with $T_\varrho' = \bar T_3$. Otherwise, if $t_1<\bar T_3 \wedge T_\eps$ we define
\[
t_0 =  \inf\{t\in [0,t_1]\colon R_s > \varrho/5 \;\;\forall\, s\in [t,t_1] \}\,.
\]
We observe that $t_0>0$ for any $\eps$ small enough since $R_0\le \eps$. Moreover, $R_{t_1} = \varrho/2$, $R_{t_0} = \varrho/5$, and $R_t \in [\varrho/5,\varrho/2]$ for any $t\in [t_0,t_1]$. In particular, $m_t(R_t/2)\le m_t(\varrho/10) \le \eps^3$ for any $t\in [t_0,t_1]$ and $\eps$ small enough. Clearly, to prove Eq.~\eqref{supp_pro} it is enough to show that there exists  $T_\varrho'\in (0,\bar T_3]$ such that $t_1-t_0 \ge T_\varrho'\wedge T_\eps$ for any $\eps$ small enough.

To this end, we notice that by Lemma \ref{lem:5.1}, if $\eps$ is small enough then $\Lambda_{i,\eps}(t) \subset \Sigma(B^{i,\eps}(t)|R(t))$ for any $t\in (t_0,t_1)$, with $R(t)$ solution to 
\begin{equation}
\label{stimrbis}
\dot R(t) =  2 C_FR(t) + \frac{|a|}{\pi\alpha} + \frac{4}{\pi|\log\eps|^\gamma R(t)^3} + g_\eps(t) \,, \quad R(t_0) = \varrho/4\,,
\end{equation}
where $g_\eps(t)$ is any smooth function which is an upper bound for the last term in Eq.~\eqref{stimv}. Indeed, this is true for $t=t_0$ and suppose, by absurd, there were a first time $t_*\in (t_0,t_1)$ such that $|x(t_*)-B^{i,\eps}(t_*)| = R(t_*)$ for some fluid particle initially located at $x(0) = x_0\in \Lambda_{i,\eps}(0)$. Then $R(t_*) = R_{t_*}$ in view of Eq.~\eqref{Rt}, and hence, by Eq~\eqref{stimv} and using that $|a_i|  < |a|$, the radial velocity of $x(t)- B^{i,\eps}(t)$ at $t=t_*$ would be strictly smaller than $\dot R(t_*)$, in contradiction with the definition of $t_*$ as the first time at which the graph of $t \mapsto |x(t)-B^{i,\eps}(t)|$ crosses the one of $t\mapsto R(t)$.

Since $m_t(R_t/2) \le \eps^3$ we can choose $g_\eps(t) \le 2\sqrt {M\eps}$ for any $t\in [t_0,t_1]$ and $\eps$ small enough. Therefore, by \eqref{stimrbis},
\begin{equation*}
\dot R(t) \le 2 C_F R(t) + \frac{|a|}{\pi\alpha} +  \frac{4}{\pi|\log\eps|^\gamma (\varrho/4)^3} + 2\sqrt{M\eps} \quad \forall\, t\in [t_0,t_1]\,,
\end{equation*}
where we also used that $R(t) \ge R(t_0)=\varrho/4$ to estimate from above the nonlinear term in the right-hand side of Eq.~\eqref{stimrbis}. This means that, e.g., $\dot R(t) \le 2C_F R(t) + |a|/\alpha$ for any $ t\in [t_0,t_1]$ and $\eps$ small enough, whence
\[
R(t_1) \le \rme^{2C_F(t_1-t_0)} R(t_0) + \frac{|a|}{2C_F \alpha}\big(\rme^{2C_F(t_1-t_0)}  -1 \big)\,,
\]
i.e.,
\begin{equation}
\label{rt12}
t_1-t_0 \ge  \frac{1}{2C_F} \log\left(\frac{2C_F\alpha R(t_1) + |a|}{2C_F\alpha R(t_0) + |a|} \right),
\end{equation}
Therefore, as $R(t_1) \ge  \varrho/2$ and $R(t_0) = \varrho/4$, the claim follows with
\[
T_\varrho' =  \frac{1}{2C_F} \log\left( \frac{4C_F\alpha\varrho + 4|a|}{2C_F\alpha\varrho + 4 |a|} \right) \wedge \bar T_3\,,
\]
where, according to what discussed in Remark \ref{rem:4.1}, we can choose
\[
\bar T_3 =  \frac{\varrho\,\rme^{-4}}{4C_W} \left(C_F \varrho +\frac{10|a|}{\pi\alpha}\right)^{-1} \wedge T\,.
\]
The proposition is proved.
\end{proof}

\section{Conclusion of the proof of Theorem \ref{thm:1}}
\label{sec:6}

In this section we prove the following limits,
\begin{gather}
\label{secl} 
\lim_{\eps \to 0} \max_{i=1,\ldots,N} \max_{t\in [0, T_\eps]}\big| B^{i,\eps}(t) - \zeta^i(t) \big| = 0\,, \\ \label{pril} \lim_{\eps\to 0}  \max_{i=1,\ldots,N} \max_{t\in [0,T_\eps]} |B^{i,\eps}(t)-q^{i,\eps}(t)| = 0\,,
\end{gather}
with $q^{i,\eps}(t)$ as in Theorem \ref{thm:conc}. This concludes the proof of the main theorem. Indeed, from Eq.~\eqref{secl} and Proposition \ref{prop:5.1} it follows by continuity that $T_\eps\ge T_\varrho'$ for any $\eps$ small enough. Therefore, in view of Eq.~\eqref{pril} and applying Theorem \ref{thm:conc} with, e.g., $R=R_\eps=\exp\sqrt{|\log\eps|}$, the statement of Theorem \ref{thm:1} is proved with $T_\varrho=T_\varrho'$, $\zeta^{i,\eps}(t) = q^{i,\eps}(t)$, and $\varrho_\eps = \eps R_\eps$.

\begin{proof}[Proof of Eq.~(\ref{secl})]
In what follows, we shall denote by $C$ a generic positive constant, whose numerical value may change from line to line. Let
\[
\Delta(t) := \sum_i |B^{i,\eps}(t) - \zeta^i(t)|^2\,, \quad t\in [0,T_\eps]\,.
\]
From Eqs.~\eqref{ode}, \eqref{bpunto}, and noticing that
\[
F^i(x,t) = \sum_{j\ne i} \left[(K*\omega_{j,\eps})(x,t) + u^j_L(x,t) + u^j_{\mc R}(x,t)\right],
\]
we have
\[
\dot \Delta(t) = 2\sum_i (B^{i,\eps}(t) - \zeta^i(t)) \cdot (\dot B^{i,\eps}(t) - \dot \zeta^i(t))  = 2\sum_{p=1}^4 \sum_i (B^{i,\eps}(t) - \zeta^i(t)) \cdot D^i_p(t)\,,
\]
where
\begin{gather*}
D^i_1(t) = \frac{1}{a_i} \sum_{j\ne i} \int\!\rmd x \! \int\!\rmd y\, \left[ K(x-y) -  K(\zeta^i(t)-\zeta^j(t))  \right] \omega_{i,\eps}(x,t)\, \omega_{j,\eps}(y,t) \,, \\ D^i_2(t) = \frac{1}{a_i} \int\!\rmd x\, u^i_L(x,t)\, \omega_{i,\eps}(x,t) - \frac{a_i}{4\pi\alpha} \begin{pmatrix} 1 \\ 0 \end{pmatrix},  \\ D^i_3(t) = \frac{1}{a_i} \sum_j \int\!\rmd x\, u^j_{\mc R}(x,t) \,\omega_{i,\eps}(x,t)  \,, \quad  D^i_4(t) = \frac{1}{a_i} \sum_{j\ne i} \int\!\rmd x \, u^j_L(x,t)  \, \omega_{i,\eps}(x,t)\,.
\end{gather*}
By Eqs.~\eqref{Kern} and \eqref{sep-disks}
\begin{align*}
|D^i_1(t)| & \le \frac{C}{\varrho^2|a_i|} \sum_{j\ne i} \int\!\rmd x \! \int\!\rmd y\, \left(|x-\zeta^i(t)| + |y-\zeta^j(t)| \right) |\omega_{i,\eps}(x,t)\, \omega_{j,\eps}(y,t)| \\ & \le \frac{C}{\varrho^2} \sum_{j\ne i} |a_j| \left(|B^{i,\eps}(t) - \zeta^i(t)| + |B^{j,\eps}(t) -\zeta^j(t)| + \sqrt{\frac{J_{i,\eps}(t)}{|a_i|}}+   \sqrt{\frac{J_{j,\eps}(t)}{|a_j|}} \right),
\end{align*}
where in the last inequality we used the Cauchy-Schwarz inequality and Eq.~\eqref{J}. Therefore,
\begin{equation}
\label{dp1}
\sum_i (B^{i,\eps}(t) - \zeta^i(t)) \cdot D^i_1(t) \le \frac{C\sqrt N|a|}{\varrho^2} \left(\Delta(t) + \sqrt{\sum_i \frac{J_{i,\eps}(t)}{|a_i|}}\sqrt{\Delta(t)} \right).
\end{equation}

Regarding $D^i_2(t)$, in view of Eqs.~\eqref{uL=w}, \eqref{wL=} and using  that each $\omega_{i,\eps}(x,t)$ is a non-negative or non-positive function,
\[
|D^i_2(t)| =  \left| \frac{1}{a_i} \int\!\rmd x\, |w^i_L(x,t)| \, \omega_{i,\eps}(x,t) - \frac{|a_i|}{4\pi\alpha}  \right|.
\]
By Eq.~\eqref{uL<},
\begin{equation}
\label{dp2_1}
\frac{1}{a_i} \int\!\rmd x\, |w^i_L(x,t)| \, \omega_{i,\eps}(x,t) \le \frac{|a_i|}{4\pi\alpha} + \frac{C_L}{|\log\eps|}\,.
\end{equation}
For a lower bound to the integral in the left-hand side above, we consider the disk
\begin{equation}
\label{sre}
\Sigma_{i,\eps}(t) := \Sigma(q^{i,\eps}(t)|\eps R_\eps)\,, \quad R_\eps := \exp(\sqrt{|\log\eps|\log|\log\eps|})\,,
\end{equation}
with center $q^{i,\eps}(t)$ as in Theorem \ref{thm:conc}. Using  Eq.~\eqref{sep-supp} and that $s \mapsto\log[(1+s)/s]$, $s>0$, is decreasing, if $x\in \Sigma_{i,\eps}(t)$ and $t\in [0,T_\eps]$ then, by definition Eq.~\eqref{wL=},
\[
|w^i_L(x,t)|  \ge \frac{\log[(1+2\eps R_\eps)/(2\eps R_\eps)]}{4\pi(r_\eps+ \bar d+\varrho)}\int_{ \Sigma_{i,\eps}(t)}\!\rmd y\, |\omega_\eps(y,t)|\,,
\]
whence
\[
\frac{1}{a_i} \int\!\rmd x\, |w^i_L(x,t)| \, \omega_{i,\eps}(x,t)  \ge \frac{\log[(1+2\eps R_\eps)/(2\eps R_\eps)]}{4\pi(r_\eps+ \bar d+\varrho)} \frac{1}{|a_i|} \left[\int_{ \Sigma_{i,\eps}(t)}\!\rmd y\,  |\omega_{i,\eps}(y,t)|\right]^2.
\]
By Theorem \ref{thm:conc}, provided $\eps$ is chosen sufficiently small in order to have $R_\eps >\exp(C_3\log|\log\eps|)$, we can apply Eq.~\eqref{eq:conc} with $R=R_\eps$ getting
\begin{equation}
\label{stmass1}
\int_{ \Sigma_{i,\eps}(t)}\!\rmd y\, |\omega_\eps(y,t)| \ge |a_i| - C_4 \sqrt{\frac{\log|\log\eps|}{|\log\eps|}}  \quad \forall\, t\in [0,T_\eps]\,.
\end{equation}
Since
\[
\left|\frac1{|\log\eps|} \log\frac {1+2\eps R_\eps}{2\eps R_\eps}  -1\right| \le 	\frac{C\log R_\eps}{|\log\eps|} = C \sqrt{\frac{\log|\log\eps|}{|\log\eps|}}  \,,
\]
and recalling $r_\eps = \alpha |\log\eps|$, we conclude that there is $C_5 = C_5(\alpha,|a|,\bar d, \varrho)>0$ such that, for any $\eps$ sufficiently small,
\begin{equation}
\label{uL>}
\frac{1}{a_i} \int\!\rmd x\, |w^i_L(x,t)| \, \omega_{i,\eps}(x,t) \ge \frac{|a_i|}{4\pi\alpha} - C_5\sqrt{\frac{\log|\log\eps|}{|\log\eps|}} \quad  \forall\, t\in [0,T_\eps]\,.
\end{equation}
By Eqs.~\eqref{dp2_1} and \eqref{uL>}, for any $\eps$ small enough,
\begin{equation}
\label{dp2}
\sum_i (B^{i,\eps}(t) - \zeta^i(t)) \cdot D^i_2(t) \le C_5\sqrt N \sqrt{\frac{\log|\log\eps|}{|\log\eps|}} \sqrt{\Delta(t)}\,.
\end{equation}

Concerning $D^i_3(t)$, by \eqref{stimR} we deduce that
\begin{equation}
\label{dp3}
\sum_i (B^{i,\eps}(t) - \zeta^i(t)) \cdot D^i_3(t) \le \frac{C_{\mc R} N^{3/2} \log|\log\eps|}{|\log\eps|} \sqrt{\Delta(t)}\,.
\end{equation}

Finally, by Eqs.~\eqref{uL=w}, \eqref{wL=}, and using again Eq.~\eqref{sep-supp} and that $s \mapsto\log[(1+s)/s]$ is decreasing, if $j \ne i$ then
\[
|u^j_L(x,t)| = |w^j_L(x,t)| \le \frac{|a_j|}{4\pi(r_\eps-\bar d-\varrho)} \log \frac {1+2\varrho}{2\varrho} \quad \forall\, x\in \Lambda_{i,\eps}(t) \quad  \forall\, t\in [0,T_\eps]\,,
\]
whence, by Eq.~\eqref{reps},  for any $\eps$ small enough,
\begin{equation}
\label{dp4}
\sum_i (B^{i,\eps}(t) - \zeta^i(t)) \cdot D^i_4(t) \le \frac{CN^{3/2}|a|}{ (\alpha|\log\eps| -\bar d-\varrho)} \log \frac {1+2\varrho}{2\varrho} \sqrt{\Delta(t)}\,.
\end{equation}

Given $\theta\in (0,1)$, by the bounds Eqs.~\eqref{dp1}, \eqref{dp2}, \eqref{dp3}, \eqref{dp4}, and applying Theorem \ref{thm:J<} (with $\gamma\in(\theta,1)$), we conclude that, for any $\eps$ small enough,
\[
\dot\Delta(t) \le  \frac{C\sqrt N|a|}{\varrho^2} \Delta(t) + \frac{1}{|\log\eps|^{\theta/2}}\sqrt{\Delta(t)} \quad \forall\, t\in [0,T_\eps]\,.
\]
Since $T_\eps\le T$ and $\Delta(0) \le 4N\eps^2$, the differential inequality above implies Eq.~\eqref{secl}.
\end{proof}

\begin{proof}[Proof of Eq.~(\ref{pril})] 
By Theorem \ref{thm:conc} with $R=R_\eps=\exp\sqrt{|\log\eps|}$,
\[
\begin{split}
|B^{i,\eps}(t) -q^{i,\eps}(t)| & \le \frac{1}{a_i}\int\!\rmd x\, |x-q^{i,\eps}(t)|\, \omega_{i,\eps}(x,t) \\ & \le \eps R_\eps + \frac{1}{a_i}\int_{\Sigma(q^{i,\eps}(t)|\eps R_\eps)^{\complement}}\!\rmd x\, |x-q^{i,\eps}(t)|\, \omega_{i,\eps}(x,t) \\ & \le  \eps R_\eps + \frac{C_4\log|\log\eps|}{|a_i|\sqrt{|\log\eps|}} |B^{i,\eps}(t)-q^{i,\eps}(t)| \\ & \qquad + \frac{1}{a_i}\int_{\Sigma(q^{i,\eps}(t)|\eps R_\eps)^{\complement}}\!\rmd x\, |x-B^{i,\eps}(t)|\, \omega_{i,\eps}(x,t)\,.
\end{split}
\]
Assuming $\eps$ so small to have $|a_i|\sqrt{|\log\eps|} \ge 2 C_4\log|\log\eps|$, we get (by applying in the end the Cauchy-Schwarz inequality)
\[
\begin{split}
|B^{i,\eps}(t) -q^{i,\eps}(t)| & \le 2  \eps R_\eps +  \frac{2}{a_i}\int_{\Sigma(q^{i,\eps}(t)|\eps R_\eps)^{\complement}}\!\rmd x\, |x-B^{i,\eps}(t)|\, \omega_{i,\eps}(x,t) \\ & \le 2 \eps R_\eps + 2 \sqrt{|a_i|J_{i,\eps}(t)}\,,
\end{split}
\]
and Eq.~\eqref{pril} follows by Theorem \ref{thm:J<}.
\end{proof}

\section{An example of leapfrogging vortex rings}
\label{sec:7}

When we have two vortex rings only, the dynamics of their centers of vorticity (in the limit $\eps\to 0$) can be completely studied, giving rise, for suitable values of the initial data, to the so called \textit{leapfrogging} dynamics, which was first described by Helmholtz \cite{H,H1}, as already discussed in the Introduction.

Although Theorem \ref{thm:1} guarantees convergence for short times only, in the special case of two vortex rings with large enough main radii we are able to extend the time of convergence in order to cover several crossings between the rings. As already noticed in the Introduction, this is completely consistent with the physical phenomenon, since the leapfrogging motion of two vortex rings is observed experimentally and numerically up to a few crossings, after which the rings dissolve and lose their shape.

Let us then describe the dynamical system Eq.~\eqref{ode} for $N=2$ and suppose that their vortex intensities satisfy $a_1+a_2\ne 0$. Adopting the new variables,
\[
x = \zeta^1-\zeta^2\,, \qquad \qquad y = \frac{a_1\zeta^1+ a_2 \zeta^2}{a_1+a_2}\,,
\]
the equations take the form
\begin{equation}
\label{xyeq}
\left\{\begin{aligned} & \dot x=-\frac{a_1+a_2}{2\pi}\nabla^\perp\log |x| +\frac{a_1-a_2}{4\pi \alpha}\begin{pmatrix} 1 \\ 0 \end{pmatrix}, \\
&\dot y =\frac{a_1^2+a_2^2}{4\pi\alpha(a_1+a_2)} \begin{pmatrix} 1 \\ 0 \end{pmatrix}. \end{aligned}\right.
\end{equation}
The barycenter $y$ performs a rectilinear uniform motion, while the evolution of the relative position $x$ is governed by the canonical equations $\dot x = \nabla^\perp \mc H(x)$ of Hamiltonian
\[
\mc H(x) = -\frac{a_1+a_2}{4\pi} \log |x|^2 +\frac{a_1-a_2}{4\pi\alpha}x_2\,,
\qquad x=(x_1, x_2)\,.
\]
Hereafter, for the sake of concreteness, we furthermore assume $a_1>|a_2|$, the other cases can be treated analogously.

\begin{figure}
\includegraphics[scale=.30,angle=0,draft=false]{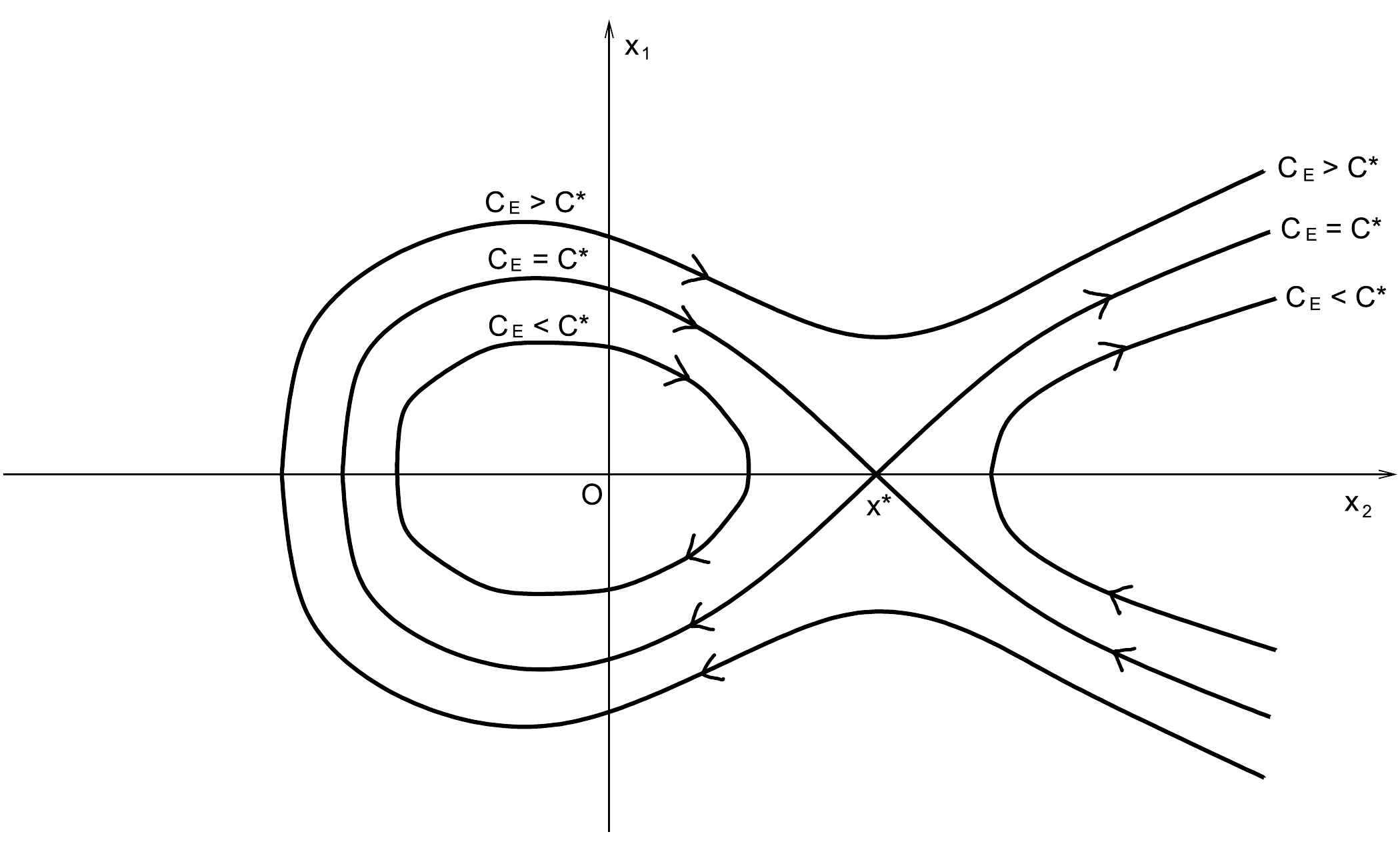}
\caption{Phase portrait of the dynamical system describing the motion of the relative position $x$ between the two rings.}
\label{fig:1}
\end{figure}

The phase portrait of this Hamiltonian system can be obtained by drawing the energy level sets $\{x\colon \mc H(x) = E\}$ (invariant sets, each one composed by a finite union of phase curves). To this end, we recast the equation $\mc H(x) = E$ in the form
\[
x_1= \pm f(x_2)\,, \quad\textnormal{with} \quad f(x_2) = \sqrt{C_E\exp\left( \frac{ x_2}{\alpha a} \right)-x_2^2}\,,
\]
where
\[
a=\frac{a_1+a_2}{a_1-a_2}\,, \qquad C_E=\exp\left(-\frac{4\pi E}{a_1+a_2} \right).
\]
There is a unique equilibrium, corresponding to the critical point $x^*=(0, 2\alpha a)$ of $\mc H$ and we set
\[
C^*=\exp\left(-\frac{4\pi \mc H(x^*)}{a_1+a_2}\right) =\left(\frac{2\alpha a}{ \rme} \right)^2.
\]
It is easily seen that
\[
\mathrm{Dom}(f) = \{ x_2 \colon |x_2| \le \sqrt{C_E} \exp(x_2/2\alpha a)\} = \begin{cases} [\eta_1,\eta_2] \cup [\eta_3, +\infty) & \text{if } C_E < C^* \\  [\bar \eta, +\infty) & \text{if } C_E \ge C^* \end{cases}
\]
with $\eta_1 < 0 < \eta_2 < x^*_2 < \eta_3$ and $\bar\eta < 0$. It follows that the phase portrait looks like qualitatively as depicted in Figure \ref{fig:1}. We notice that one ring overtakes the other when $x_1 = 0$ and $\dot x_1 \ne 0$. In particular, the periodic motions occurring for $0<C_E<C^*$ (whose orbits are the closed curves in Figure \ref{fig:1}), correspond to the leapfrogging behavior, in which the rings pass through each other alternately.

Since along the orbit we have
\[
\dot x_2 =  \frac{a_1+a_2}{2\pi}\frac{x_1}{|x|^2} = \pm\frac{(a_1+a_2)\, \rme^{-x_2/\alpha a} }{2\pi C_E} \sqrt{C_E \, \rme^{x_2/\alpha a} - x_2^2}\,,
\]
the period of a close orbit on the level $C_E<C^*$ is given by
\[
T_E = 2\int_{\eta_1}^{\eta_2}\! \frac{\rmd x_2}{|\dot x_2|} = \frac{4\pi C_E}{a_1+a_2} \int_{\eta_1}^{\eta_2} \! \rmd x_2\, \frac{\rme^{x_2/\alpha a}}{\sqrt{C_E\,\rme^{x_2/\alpha a} - x_2^2}}\,,
\]
with $\eta_1<0<\eta_2$ as before, i.e., the two smallest roots of the equation
$C_E \rme^{x_2/\alpha a}-x_2^2=0$. We also note that, for small values of the positive constant $C_E$,  we have $\eta_{1,2}\approx \mp \sqrt{C_E}$, and
\[
T_ E \approx \frac{4\pi C_E}{a_1+a_2}\int_{-\sqrt{C_E}}^{\sqrt{C_E}}\! \frac{ \rmd x_2}{\sqrt{C_E - x_2^2}} = \frac{4\pi^2 C_E}{a_1+a_2}\,,
\]
which goes to $0$ as $C_E\to 0$ (i.e., $E\to +\infty$, note that $\mc H(x)$ diverges as $x\to 0$).

The time threshold $T_\varrho$ in Theorem \ref{thm:1} can be seen to be bounded by a constant multiple of $C_F^{-1}$ (recall that $T_\varrho = T_\varrho'$ with $T_\varrho'$ as in Proposition \ref{prop:5.1}). On the other hand, $C_F$ is an upper bound for the velocity field (and its Lipschitz constant) produced by one ring and acting on the second one, so it depends on the distance between the centers of vorticity of the two rings as a constant multiple of $(|a_1|+|a_2|)/\varrho^2$ (at short distances), and $\varrho$ is of order $|\eta_{1,2}|$ in the periodic motion considered above. Therefore, $T_\varrho$ and $T_E$ are of the same order also when $T_E$ is small, and a direct inspection easily shows that $T_\varrho <T_E$. Thus, a mere application of Theorem \ref{thm:1} guarantees at most one overtaking between the rings during the time interval $[0,T_\varrho]$ (to this end, it is enough to choose the initial data on the orbit close enough to the point $(0,\eta_1)$ or $(0,\eta_2)$). 

We next show that the result can be improved in the case of rings with large main radii, i.e., when the parameter $\alpha$ is chosen large enough (with respect to the distance between the centers of vorticity). 

The key observation is that when $\alpha\to +\infty$ the system Eq.~\eqref{xyeq} reduces to the standard planar vortex model, i.e., $\dot y=0$ and $\mc H(x) = -\frac{a_1+a_2}{4\pi} \log |x|^2$, so that each level set $\{x\colon \mc H(x) = E\}$ consists of a circular orbit traveled at constant speed, with period
\begin{equation}
\label{period2}
\mc T_E = \frac{4\pi^2 R_E^2}{a_1+a_2}\,,
\end{equation}
where $R_E= \sqrt{C_E} = \exp\big[-2\pi E/(a_1+a_2)\big]$ is the radius of the orbit (note that $C^*\to +\infty$ as $\alpha\to + \infty$).
We omit the proof of the Lemma \ref{lem:alphagrande} below, which easily follows from the previous observation and standard arguments in the theory of ordinary differential equations.

\begin{lemma}
\label{lem:alphagrande}
Given $\varrho>0$, fix $E>0$ such that $R_E>4\varrho$, an integer $k\in \bb N$, and let $T = (k+1) \mc T_E$ with $\mc T_E$ as in Eq.~\eqref{period2}. Then there exits $\alpha_0>0$ such that for any $\alpha \ge \alpha_0$ we have $C_E < C^*$  and the corresponding periodic motion $t\mapsto x_E(t)$ solution to Eq.~\eqref{xyeq}$_a$ satisfies
\begin{equation}
\label{T1}
\min_{t\in [0,T]} |x_E(t)| \ge 4\varrho\,, \qquad kT_E < T\,.
\end{equation}
\end{lemma}

In the sequel, we fix a solution $t\mapsto (\zeta^1(t),\zeta^2(t))$ in such a way that $\zeta^1(t) - \zeta^2(t) = x_E(t)$, with $x_E(t)$ as in Lemma \ref{lem:alphagrande}. Therefore, in view of Eq.~\eqref{T1}, if we choose $\varrho$ and $T$ as in the aforementioned lemma then Eq.~\eqref{T} holds in this case for any $\alpha\ge \alpha_0$. Moreover, from the expression of $\dot y$ in Eq.~\eqref{xyeq}, the parameter $\bar d$ in Eq.~\eqref{dbarra} is uniformly bounded for $\alpha\ge \alpha_0$. Taking advantage of this uniformity, we now show that if $\alpha$ is large enough the proof of Theorem \ref{thm:1} can be improved, pushing the time threshold $T_\varrho$ up to $T$, so that at least $2k$ overtakings between the rings take place during the time interval of convergence.

We fix an integer $n \gg 1$ to be specified later and let $\alpha = \alpha_n := \alpha_0 n$. The strategy develops according to the following steps.

\medskip
\noindent
\textit{Step 0.} Letting $\bar T_3$ be as in Proposition \ref{prop:2} with the choice $R = \varrho_n := \varrho/n$, we can argue as done in Proposition \ref{prop:5.1} to deduce that there is $T_0' \in (0,T]$ such that 
\begin{equation}
\label{st1}
\Lambda_{i,\eps}(t) \subset \Sigma(B^{i,\eps}(t)|3\varrho_n) \quad \forall\, t\in [0, T_0' \wedge T_\eps]\,.
\end{equation}
To this end, we adapt the proof of that proposition by defining, in this case,
\[
t_1 := \sup\{t\in [0,\bar T_3 \wedge T_\eps] \colon R_s \le 3\varrho_n \;\, \forall s \in [0,t]\}\,,
\]
and (whenever $t_1<\bar T_3 \wedge T_\eps$)
\[
t_0 =  \inf\{t\in [0,t_1]\colon R_s > \varrho_n \;\;\forall\, s\in [t,t_1] \}\,.
\]
Therefore, choosing now $R(t_0)= 2\varrho_n$ in Eq.~\eqref{stimrbis}, from Eq.~\eqref{rt12} we deduce that Eq.~\eqref{st1} holds with
\[
T_0' = \frac{1}{2C_F} \log\left(\frac{6C_F\alpha_n\varrho_n + |a|}{4C_F\alpha_n\varrho_n + |a|} \right) \wedge \bar T_3 = \frac{1}{2C_F} \log\left(\frac{6C_F\alpha_0\varrho + |a|}{4C_F\alpha_0\varrho + |a|} \right) \wedge \bar T_3\,,
\]
where, in view of Remark \ref{rem:4.1},
\[
\bar T_3 = \frac{\varrho_n\,\rme^{-4}}{4C_W} \left(C_F \varrho_n +\frac{|a|}{\pi\alpha_n}\right)^{-1} \wedge T = \frac{\varrho\,\rme^{-4}}{4C_W} \left(C_F \varrho +\frac{|a|}{\pi\alpha_0}\right)^{-1} \wedge T\,.
\]

\medskip
\noindent
\textit{Step 1.} If $T_0'=T$ we are done, otherwise, from Step 0 and Eqs.~\eqref{secl}, \eqref{pril} we have $T_\eps>T_0'$ for any $\eps$ small enough, and whence
\[
\Lambda_{i,\eps}(T_0') \subset \Sigma(B^{i,\eps}(T_0')|3\varrho_n)\,.
\]
Then, we can adapt to the present context the arguments of Propositions \ref{prop:1} and \ref{prop:2} to deduce that, for any $\eps$ small enough,
\begin{equation}
\label{4rhon}
m^i_t(4\varrho_n) \le \eps^{\ell} \quad \forall\, t\in[T_0', (T_0' + \bar T_\ell)  \wedge T_\eps]\,.
\end{equation}
More precisely:

(i) We follow the proof of Proposition \ref{prop:1}, with $R-\varrho_n/4$ in place of $R/2$ in the computations leading to Eq.~\eqref{2mass 4''}, and iterate Eq.~\eqref{mass 14'} from $3\varrho_n + \varrho_n/2 -h$ to $3\varrho_n + \varrho_n/4$, getting in this way $m^i_t(3\varrho_n+\varrho_n/2) \le |\log\eps|^{-2}$ for $t\in[0, (T_0' + \widetilde T_2)  \wedge T_\eps]$. Moreover, since in this case $R = 3\varrho_n + \varrho_n/2$ and $h=\varrho_n/(4n+4)$,
\[
\widetilde T_2 =  \frac{\varrho_n\rme^{-3}}{8C_W} \left(C_F \left(3\varrho_n+\frac{\rho_n}2\right) +\frac{|a|}{\pi\alpha_n}\right)^{-1} \wedge (T-T_0')\,.
\]

(ii) Using (i), we can adapt the proof of Proposition \ref{prop:2}, iterating now from $4\varrho_n -h$ to $4\varrho_n - \varrho_n/4$. Recalling Remark \ref{rem:4.1} (adapted to the present context, in particular with $\widetilde T_2$ as above), Eq.~\eqref{4rhon} for $\ell>2$ thus holds with
\[
\begin{split}
\bar T_\ell & =  \frac{\varrho_n\rme^{- \ell -1}}{8C_W} \left(C_F 4\varrho_n +\frac{|a|}{\pi\alpha_n}\right)^{-1} \wedge (T-T_0') \\ & = \frac{\varrho\rme^{- \ell -1}}{8C_W} \left(C_F 4\varrho +\frac{|a|}{\pi\alpha_0}\right)^{-1} \wedge  (T-T_0')\,.
\end{split}
\]

Using Eq.~\eqref{4rhon}, we can now adjust the reasoning of Section \ref{sec:5} to prove that, for any $\eps$ small enough,
\begin{equation}
\label{st2}
\Lambda_{i,\eps}(t) \subset \Sigma(B^{i,\eps}(t)|6\varrho_n) \quad \forall\, t\in [T_0', (T_0'+T_1') \wedge T_\eps]\,,
\end{equation}
with $T_1' \in (0,T-T_0']$ as detailed below. More precisely:

(i) We modify the claim of Lemma \ref{lem:5.1} by replacing Eq.~\eqref{stimv} with
\[
\frac{\rmd}{\rmd t} |x(t)- B^{i,\eps}(t)| \le 2 C_F R_t + \frac{|a_i|}{\pi\alpha} +  \frac{6}{\pi|\log\eps|^\gamma (R_t\wedge \varrho_n)^3} + \sqrt{\frac{M m^i_t(R^n_t)}{\eps^2}}\,,
\]
where $R^n_t = (R_t-(\varrho_n/2)) \vee (R_t/2)$. To this end, it is enough to change the proof of Lemma \ref{lem:5.1} by splitting $(K*\omega_{i,\eps})(x,t)$ as in Eq.~\eqref{in A_1,A_2} but choosing now $\mc D=\Sigma(B^{i,\eps}(t)|R^n_t)$ and $\mc A= \Sigma(B^{i,\eps}(t)|R_t)\setminus\Sigma(B^{i,\eps}(t)|R^n_t)$. We omit the details.

(ii) Letting $\bar T_3$ be as in Eq.~\eqref{4rhon} for $\ell=3$, we prove Eq.~\eqref{st2} following the proof of Proposition \ref{prop:5.1} by defining, in this case,
\[
t_1 := \sup\{t\in [T_0',(T_0'+\bar T_3) \wedge T_\eps] \colon R_s \le 6\varrho_n \;\, \forall s \in [0,t]\}\,,
\]
and (whenever $t_1<(T_0'+\bar T_3) \wedge T_\eps$)
\[
t_0 =  \inf\{t\in [T_0',t_1]\colon R_s > 4\varrho_n+\varrho_n/2 \;\;\forall\, s\in [t,t_1] \}\,.
\]
We remark that if $t\in [t_0,t_1]$ then $m^i_t(R^n_t) \le m^i_t(4\varrho_n) \le \eps^3$ by Eq.~\eqref{4rhon}. Therefore, choosing now $R(t_0)= 5\varrho_n$ in Eq.~\eqref{stimrbis}, from Eq.~\eqref{rt12} we deduce that Eq.~\eqref{st2} holds with
\[
T_1' = \frac{1}{2C_F} \log\left(\frac{12C_F\alpha_n\varrho_n + |a|}{10C_F\alpha_n\varrho_n + |a|} \right) \wedge \bar T_3 = \frac{1}{2C_F} \log\left(\frac{12C_F\alpha_0\varrho + |a|}{10C_F\alpha_0\varrho + |a|} \right) \wedge \bar T_3\,,
\]
where
\[
\bar T_3 =  \frac{\varrho\,\rme^{-4}}{8C_W} \left(C_F 4\varrho +\frac{|a|}{\pi\alpha_0}\right)^{-1} \wedge  (T-T_0')\,.
\]

\medskip
\noindent
\textit{Step 2.} If $T_0'+T_1' =T$ we are done, otherwise from Eqs.~\eqref{secl}, \eqref{pril}, and \eqref{st2} we have $T_\eps>T_0'+T_1'$ for any $\eps$ small enough, and whence
\[
\Lambda_{i,\eps}(T_0'+T_1') \subset \Sigma(B^{i,\eps}(T_0'+T_1')|6\varrho_n)\,.
\]
Therefore, analogously to what done in the Step 1, this implies that, for any $\eps$ small enough,
\[
m^i_t(7\varrho_n) \le \eps^{\ell} \quad \forall\, t\in[T_0'+T_1', (T_0' +T_1'+ \bar T_\ell)  \wedge T_\eps]\,,
\]
with
\[
\bar T_\ell = \frac{\varrho\rme^{- \ell -1}}{8C_W} \left(C_F 7\varrho +\frac{|a|}{\pi\alpha_0}\right)^{-1} \wedge  (T-T_0'-T_1')\,,
\]
whence
\[
\Lambda_{i,\eps}(t) \subset \Sigma(B^{i,\eps}(t)|9\varrho_n) \quad \forall\, t\in [T_0'+T_1', (T_0'+T_1'+T_2') \wedge T_\eps]\,,
\]
with
\[
T_2' = \frac{1}{2C_F} \log\left(\frac{18C_F\alpha_n\varrho_n + |a|}{16C_F\alpha_n\varrho_n + |a|} \right) \wedge \bar T_3 = \frac{1}{2C_F} \log\left(\frac{18C_F\alpha_0\varrho + |a|}{16C_F\alpha_0\varrho + |a|} \right) \wedge \bar T_3\,,
\]
and
\[
\bar T_3 =  \frac{\varrho\,\rme^{-4}}{8C_W} \left(C_F 7\varrho +\frac{|a|}{\pi\alpha_0}\right)^{-1} \wedge  (T-T_0'-T_1')\,.
\]

\medskip
\noindent
\textit{Step $j$.}
The above procedure can be iterated inductively in the following manner. If at the end of the $(j-1)$th step we still have $T_0'+ \cdots + T_{j-1}' < T$ (otherwise we are done) and $3j\varrho_n < \varrho$ then $T_\eps>T_0'+ \cdots + T_{j-1}'$ for any $\eps$ small enough, so that
\[
\Lambda_{i,\eps}(T_0'+ \cdots + T_{j-1}') \subset \Sigma(B^{i,\eps}(T_0'+ \cdots + T_{j-1}')|3j \varrho_n)\,,
\]
which allows for a further iteration, giving first
\[
m^i_t((3j+1)\varrho_n) \le \eps^{\ell} \quad \forall\, t\in[T_0' + \cdots + T_{j-1}' , (T_0' + \cdots + T_{j-1}' + \bar T_\ell)  \wedge T_\eps]\,,
\]
with
\[
\bar T_\ell = \frac{\varrho\rme^{- \ell -1}}{8C_W} \left(C_F (3j+1) \varrho +\frac{|a|}{\pi\alpha_0}\right)^{-1} \wedge  [T-(T_0'+\cdots +T_{j-1}')]\,,
\]
and then
\[
\Lambda_{i,\eps}(t) \subset \Sigma(B^{i,\eps}(t)|(3j+3)\varrho_n) \quad \forall\, t\in [T_0'+\cdots +T_{j-1}', (T_0'+ \cdots + T_j') \wedge T_\eps]\,,
\]
with 
\[
T_j' = \frac{1}{2C_F} \log\left(\frac{2(3j+3)C_F\alpha_0\varrho + |a|}{2(3j+2) C_F\alpha_0\varrho + |a|} \right) \wedge \bar T_3
\]
and
\[
\bar T_3 =  \frac{\varrho\,\rme^{-4}}{8C_W} \left(C_F (3j+1) \varrho +\frac{|a|}{\pi\alpha_0}\right)^{-1} \wedge  [T-(T_0'+ \cdots + T_{j-1}')]\,.
\]

\medskip
\noindent
\textit{Conclusion.} The maximum number of possible iterations is given by $j_* = j_n \wedge j_T$, where
\begin{gather*}
j_n = \max\{j\colon 3(j+1)\varrho_n \le \varrho/2\} = \left\lfloor \frac n 6 -1\right\rfloor, \\
j_T = \max\{ 0<j \le j_n \colon T_0'+T_1'+ \ldots + T_{j-1}' < T\}\,.
\end{gather*}
From the explicit expression of $T_j'$, if $j<j_T$ then
\[
T_j' = A_j := \frac{1}{2C_F} \log\left(\frac{2(3j+3)C_F\alpha_0\varrho + |a|}{2(3j+2) C_F\alpha_0\varrho + |a|} \right) \wedge  \frac{\varrho\,\rme^{-4}}{8C_W} \left(C_F (3j+1) \varrho +\frac{|a|}{\pi\alpha_0}\right)^{-1}.
\]
Since $A_j = O(j^{-1})$ for $j$ large, whence $\sum_{j=0}^{j_n} A_j = O(\log j_n) = O(\log n)$, by choosing $n$ (i.e., $\alpha=\alpha_n$) large enough we get $j_*< j_n$, which means $\sum_{j=0}^{j_*} T_j' = T$, i.e., the convergence holds up to the chosen time $T>kT_E$.
\qed

\appendix

\section{Proof of Lemma \ref{lem:Hdec}}
\label{app:a}

Eq.~\eqref{stimH} easily follows from Eqs.~\eqref{H1}, \eqref{H2}, and \eqref{sep-disks}, we omit the details. Concerning the decomposition Eq.~\eqref{sH}, we observe that
\[
\begin{split}
H(x,y) & = - \frac{1}{2\pi(r_\eps+x_2)} I_1\left(\frac{|x-y|}{\sqrt{A}}\right) \frac{(x-y)^\perp }{\sqrt{A}} \\ & \quad + \frac{1}{2\pi(r_\eps+x_2)} I_2\left(\frac{|x-y|}{\sqrt{A}}\right)\sqrt{\frac{r_\eps+y_2}{r_\eps+x_2}} \begin{pmatrix} 1 \\ 0 \end{pmatrix},
\end{split}
\]
where, for any $s>0$,
\[
I_1(s) = \int_0^\pi\!\rmd\theta\, \frac{\cos\theta}{[s^2 + 2(1-\cos\theta)]^{3/2}}\,, \quad I_2(s) = \int_0^\pi\!\rmd\theta\, \frac{1-\cos\theta}{[s^2+2(1-\cos\theta)]^{3/2}}\,.
\]
By an explicit computation, see, e.g., the Appendix in \cite{Mar99}, for any $s>0$,
\[
I_1(s) = \frac{1}{s^2} + \frac 14 \log\frac{s}{1+s} + \frac{c_1(s)}{1+s}, \quad I_2(s) = -\frac 12 \log\frac{s}{1+s} + \frac{c_2(s)}{1+s},
\]
with $c_1(s)$, $c_2(s)$ uniformly bounded for $s\in (0,+\infty)$. Therefore, the  kernel $\mc R(x,y)$ defined by \eqref{sH} is given by
\[
\mc R (x,y) = \sum_{j=1}^6 \mc  R^j(x,y),
\]
with, letting $a = |x-y|/\sqrt A$, 
\begin{align*}
R^1(x,y) & = \frac{1}{2\pi} \bigg(1 - \sqrt{\frac{r_\eps+y_2}{r_\eps+x_2}} \bigg) \frac{(x-y)^\perp }{|x-y|^2}\,, \\ R^2(x,y) & =  \frac{1}{8\pi} \bigg(\log\frac{1+a}{a}\bigg) \frac{(x-y)^\perp }{(r_\eps+x_2)\sqrt{A}}\,, \\ R^3(x,y) & = \frac{1}{4\pi (r_\eps+x_2)} \sqrt{\frac{r_\eps+y_2}{r_\eps+x_2}} \bigg(\log\frac{|x-y|}{1+|x-y|} - \log\frac{a}{1+a}\bigg) \begin{pmatrix} 1 \\ 0 \end{pmatrix}, \\ R^4(x,y) & = \frac{1}{4\pi (r_\eps +x_2)} \bigg(1-\sqrt{\frac{r_\eps+y_2}{r_\eps+x_2}} \bigg) \log\frac{|x-y|}{1+|x-y|} \begin{pmatrix} 1 \\ 0 \end{pmatrix} , \\ R^5(x,y)  & = -\frac{c_1(a)}{2\pi(1+a)} \frac{(x-y)^\perp }{(r_\eps+x_2)\sqrt{A}}\,, \\ R^6(x,y) & = \frac{c_2(a)}{2\pi(1+a)(r_\eps+x_2)}\sqrt{\frac{r_\eps+y_2}{r_\eps+x_2}} \begin{pmatrix} 1 \\ 0 \end{pmatrix}.
\end{align*}
Using that
\[
\bigg|1-\sqrt{\frac{r_\eps+y_2}{r_\eps+x_2}} \bigg| = \frac{|y_2-x_2|}{r_\eps+x_2+\sqrt{A}} \le \frac{|x-y|}{r_\eps+x_2}
\]
and
\[
\bigg|\log\frac{|x-y|}{1+|x-y|} - \log\frac{a}{1+a}\bigg| = \bigg| \log\frac{1+a}{A^{-1/2}+a}\bigg| \le \frac 12 |\log A|,
\]
we have,
\begin{align*}
& |R^1(x,y)| = \frac{1}{2\pi} \bigg|1-\sqrt{\frac{r_\eps+y_2}{r_\eps+x_2}}\bigg| \frac{1}{|x-y|}\le \frac{1}{2\pi (r_\eps+x_2)}, \\
& |R^2(x,y)| = \frac{1}{8\pi (r_\eps+x_2)} \bigg(\log\frac{1+a}{a}\bigg) \frac{|x-y|}{\sqrt{A}} \le \frac{1}{8\pi (r_\eps+x_2)} \, \sup_{s>0}\bigg( s\log\frac{1+s}{s}\bigg), \\
& |R^3(x,y)| + |R^6(x,y)| \le  \frac{1}{4\pi (r_\eps+x_2)} \sqrt{\frac{r_\eps+y_2}{r_\eps+x_2}} \bigg(|\log A| + \sup_{s>0} \frac{2c_2(s)}{1+s}\bigg), \\
& |R^4(x,y)| = \frac{1}{4\pi (r_\eps+x_2)} \bigg|1 - \sqrt{\frac{r_\eps+y_2}{r_\eps+x_2}} \bigg| \log\frac{1+|x-y|}{|x-y|} \\ & \hskip1.5cm \le \frac{1}{4\pi (r_\eps+x_2)^2} \,\sup_{s>0}\bigg( s\log\frac{1+s}{s}\bigg), \\
& |R^5(x,y)| = \frac{|c_1(a)|}{2\pi (1+a)} \frac{|x-y|}{(r_\eps+x_2)\sqrt A} \le \frac{1}{2\pi (r_\eps+x_2)} \, \sup_{s>0}\frac{s c_1(s)}{1+s}.
\end{align*}
In conclusion,
\[
\begin{split}
& |R^1(x,y)| + |R^2(x,y)| + |R^5(x,y)| \le \frac{C}{r_\eps+x_2}, \quad |R^4(x,y)| \le \frac{C}{(r_\eps+x_2)^2}, \\ & |R^3(x,y)| + |R^6(x,y)| \le  \frac{C}{r_\eps+x_2} \sqrt{\frac{r_\eps+y_2}{r_\eps+x_2}} \bigg(1+|\log A|\bigg).
\end{split}
\]
The lemma is thus proven.
\qed

\end{document}